\documentclass[11pt,reqno,tbtags]{amsart} 
\usepackage[
    colorlinks=true,%
    breaklinks,
    hyperindex,%
    plainpages=false,%
    bookmarksopen=false,%
    linkcolor=blue,%
    anchorcolor=blue,%
    citecolor=blue,%
    filecolor=black,%
    menucolor=black,%
    urlcolor=blue,%
    pdfview=FitH,
    pdfstartview=FitH,
    hyperfigures,
    bookmarksnumbered%
  ]{hyperref} 
\usepackage[numeric,initials,nobysame]{amsrefs}
\usepackage{amsmath,amssymb,amsthm,amsfonts,amsbsy,latexsym,graphicx,dsfont}
\usepackage{upref,setspace,enumerate,color,bbm,soul} 
\usepackage{pdfsync,multirow}          
\usepackage[portrait,margin=3cm]{geometry}  
\newcommand{\change}[1]{{\color{black}#1}} 
\newtheorem{thm}{Theorem}[section]
\newtheorem{theorem}{Theorem}[section]
\newtheorem{lem}[thm]{Lemma}
\newtheorem{lemma}[thm]{Lemma}
\newtheorem{cor}[thm]{Corollary}
\newtheorem{corollary}[thm]{Corollary}
\newtheorem{prop}[thm]{Proposition}
\newtheorem{proposition}[thm]{Proposition}
\newtheorem{defn}[thm]{Definition}

\theoremstyle{remark} 
\newtheorem{rem}[thm]{Remark}

\theoremstyle{definition}
\newtheorem{ass}[thm]{Assumption}

\newtheorem{conj}[thm]{Conjecture}
   
\numberwithin{equation}{section}
\numberwithin{figure}{section}
\numberwithin{table}{section}

\newenvironment{enumerateA}{\begin{enumerate}[\upshape A.]}{\end{enumerate}}

\newcommand{\ind}{\mathds{1}}
\newcommand{\eps}{\varepsilon}

\newcommand{\norm}[1]{\left\Vert#1\right\Vert}
\newcommand{\abs}[1]{\left\vert#1\right\vert}

\newcommand{\goesto}{\longrightarrow}
\newcommand{\Real}{\mathds{R}}

\newcommand{\ie}{\emph{i.e.,}}

\newcommand{\equald}{\stackrel{\mathrm{d}}{=}}
\newcommand{\probc}{\stackrel{\mathrm{P}}{\longrightarrow}}
\newcommand{\weakc}{\Longrightarrow}

\def\qed{\hfill$\square$}  
\let\ga=\alpha \let\gb=\beta \let\gc=\gamma \let\gd=\delta 
     \let\gl=\lambda        \let\go=\omega  \let\gr=\rho \let\gs=\sigma  
  
 \let\gD=\Delta

\newcommand{\cF}{\mathcal{F}}
\newcommand{\cI}{\mathcal{I}}

\newcommand{\cP}{\mathcal{P}}

\newcommand{\vzero}{\mathbf{0}}

\newcommand{\vx}{\mathbf{x}}
\newcommand{\vy}{\mathbf{y}}\newcommand{\vz}{\mathbf{z}} 
\newcommand{\mvzero}{\boldsymbol{0}}

\newcommand{\mva}{\boldsymbol{a}}

\newcommand{\mve}{\boldsymbol{e}}

\newcommand{\mvx}{\boldsymbol{x}}\newcommand{\mvy}{\boldsymbol{y}}

\newcommand{\mvgo}{\boldsymbol{\omega}}

\newcommand{\dR}{\mathds{R}}

\newcommand{\dZ}{\mathds{Z}} 

\DeclareMathOperator{\E}{\mathds{E}}
\DeclareMathOperator{\pr}{\mathds{P}}

\DeclareMathOperator{\var}{Var}

\begin{document}
    
\title[First-passage percolation across thin cylinders]{Central limit theorem for first-passage percolation time across thin cylinders}
\author{Sourav Chatterjee}
\author{Partha S. Dey}

\address{ 251 Mercer Street, Courant Institute of Mathematical Sciences,\hfill\break
\hspace*{4mm}New York University, New York, NY 10012-1185\hfill\break
\hspace*{4mm}Email: {\tt sourav@cims.nyu.EDU}\hfill\break
\hspace*{15.8mm}{\tt partha@cims.nyu.EDU}\hfill\break
\hspace*{4mm}Web: {\tt http://www.cims.nyu.edu/\string~sourav}\hfill\break
\hspace*{13.7mm}{\tt http://www.cims.nyu.edu/\string~partha}}
 
\date{\today}
\thanks{}
   

\begin{abstract}
We prove that first-passage percolation times across thin cylinders of the form $[0,n]\times [-h_n,h_n]^{d-1}$ obey Gaussian central limit theorems as long as $h_n$ grows slower than $n^{1/(d+1)}$. It is an open question as to what is the fastest that $h_n$ can grow so that a Gaussian CLT still holds. Under the natural but unproven assumption about existence of fluctuation and transversal exponents, and strict convexity of the limiting shape in the direction of $(1,0,\ldots,0)$,  we prove that in dimensions $2$ and $3$ the CLT holds all the way up to the height of the unrestricted geodesic. We also  provide some numerical evidence in support of the conjecture in dimension $2$.
\end{abstract}
\maketitle

\tableofcontents
 
\section{Introduction} 
\label{sec:introduction}
Before stating our theorems, let us begin with a short review of the first-passage percolation model and some of the known results.
\subsection{The model}
 More than forty years ago, Hammersley and Welsh~\cite{hw65} introduced first-passage percolation to model the spread of fluid through a randomly porous media. The standard first-passage percolation model on the $d$-dimensional square lattice $\dZ^d$ is defined as follows. Consider the edge set $E$ consisting of nearest neighbor edges, that is, $(\mvx,\mvy)\in\dZ^d\times \dZ^d$ is an edge if and only if $\norm{\mvx-\mvy}:=\sum_{i=1}^d |x_i-y_i|=1$. With each edge (also called a bond) $e\in E$ is associated an independent nonnegative random variable $\go_e$ distributed according to a fixed distribution $F$.  The random variable $\go_e$ represents the amount of time needed to pass through the edge~$e$. 

For a path $\cP$ (which will always be finite and nearest neighbor) in $\dZ^d$ define 
\[
\go(\cP):=\sum_{e\in\cP}\go_e 
\]                          
as the passage time for $\cP$.
For $\mvx,\mvy\in \dZ^d$, let $a(\mvx,\mvy)$, called the \emph{first-passage time}, be the minimum passage time over all paths from $\mvx$ to $\mvy$. Intuitively $a(\mvx,\mvy)$ is the first time the fluid will appear at $\mvy$ if a source of water is introduced at the vertex $\mvx$ at time $0$. Formally
\[
a(\mvx,\mvy):=\inf\{\go(\cP)\mid \cP \text{ is a path connecting } \mvx \text{ to } \mvy \text{ in } \dZ^d\}.
\]
 The principle object of study in first-passage percolation theory is the asymptotic behavior of $a(\mvzero,n\mvx)$ for fixed $\mvx\in\dZ^d$. We refer the reader to  Smythe and Wierman~\cite{sw78} and Kesten~\cite{kes86} for \change{earlier} surveys of the subject.

\subsection{Limit shape}
The first result proved by Hammersley and Welsh \cite{hw65} was that the limit 
\begin{equation}\label{hw}
\nu(\mvx):=\lim_{n\to\infty} \frac{1}{n}\E[a(0,n\mvx)]
\end{equation}
exists and is finite when $\E[\go]<\infty$ where $\go$ is a generic random variable from the distribution $F$. Moreover results of Kesten~\cite{kes86} show that $\nu(\mvx)>0$ if and only if $F(0)<p_c(d)$ where $p_c(d)$ is the  critical probability for standard bernoulli bond percolation in~$\dZ^d$. 

First-passage percolation is often regarded as a stochastic growth model by considering the growth of the random set  
\[
B_t:=\{\mvx\in\dZ^d\mid a(0,\mvx)\le t \}.
\]                                      
When $F(0)=0$, $a(\cdot,\cdot)$ is a random metric on $\dZ^d$ and $B_t$ is the ball of radius $t$ in this metric. Moreover, if $F(0)<p_c(d)$ and  $\E[\go^2]<\infty$ (or under weaker conditions in Cox and Durrett~\cite{cd81}), the growth of $B_t$ is linear in $t$ with a deterministic limit shape, that is, as $t\rightarrow \infty$, $B_t\approx tB_0\cap \dZ^d$ for a nonrandom compact set $B_0$. Precisely, the shape theorem says that (see Richardson~\cite{rson73}, Cox and Durrett \cite{cd81} and Kesten~\cite{kes86}), if $F(0)<p_c(d)$ and $\E[\min\{\go_1^{d},\go_2^d,\ldots,\go_{2d}^d\}]<\infty$ where $\go_1,\ldots,\go_{2d}$ are i.i.d.~from $F$, there is a nonrandom compact set $B_0$ such that for all $\eps>0$
\[
	(1-\eps)B_0\subseteq t^{-1}\tilde{B}_t \subseteq (1+\eps)B_0 \text{ eventually with probability one } 
\]                             
where $\tilde{B}_t=\{\mvy\in\dR^d\mid \exists\ \mvx \in B_{t} \text{ s.t. } \norm{\mvx-\mvy}\le 1\}$ is the ``inflated'' version of $B_t$.

\subsection{Tail bounds and limit theorems}
The next natural question is about the tail behavior and distributional convergence of the random variables $a(\mvzero,n\mvx)$ as $\mvx$ remains fixed and $n \rightarrow \infty$. Kesten~\cite{kes93} used martingale methods to prove that $\pr(|a(\mvzero,n\mve_1)-\E[a(\mvzero,n\mve_1)]|\ge t\sqrt{n})\le c_1e^{-c_2 t}$ for all $t\le c_3 n$ for some constants $c_i>0$, where $\mve_1$ is the unit vector $(1,0,\ldots,0)$. Later, Talagrand~\cite{tala95} used his famous isoperimetric inequality to prove that
\[
  \pr(|a(\mvzero,n\mvx)-M]|\ge t\sqrt{n\norm{\mvx}})\le c_1e^{-c_2 t^2}  
\]  
for all $t\le c_3 n$ for some constants $c_i>0$ where $M$ is a median of $a(\mvzero,n\mvx)$ and $\mvx\in\dZ^d$. Both these results were proved for distributions $F$ having finite exponential moments and satisfying $F(0)<p_c(d)$. 

From these inequalities, one might na\"ively expect that a central limit theorem holds for $a(\mvzero,n\mvx)$. However, the situation is probably much more complex, and it may not be true that a Gaussian CLT holds.  For critical first-passage percolation (assuming $F(0)=1/2$ and $F$ has bounded support) in two dimensions a Gaussian CLT was proved by Newman and Zhang~\cite{nz97}. However, this is sort of a degenerate case since here $\E[a(\mvzero,n\mvx)]$ and $\var(a(\mvzero,n\mvx))$ are both of order $\log n$ (see Chayes, Chayes and Durrett~\cite{ccd86}, and Newman and Zhang~\cite{nz97}).  When $F(0)<1/2$, we do not know of any distributional convergence result in any dimension. 

Convergence to the  Tracy-Widom law is known for {\it directed} last-passage percolation in $\dZ^2$ under very special conditions (see Subsection~\ref{sub:lit} for details), but the techniques do not carry over to the undirected case. Naturally, one may expect that convergence to something like the Tracy-Widom distribution may hold for undirected first-passage percolation also, but surprisingly, this does not seem to be the case. In the following subsection, we present our main result: a Gaussian CLT for undirected first-passage percolation when the paths are restricted to lie in thin cylinders. This gives rise to an interesting question: as the cylinders become thicker, when does the CLT break down, if it does?
 
\subsection{Our results} 
\label{sec:statement_of_results}
We consider first-passage percolation on $\dZ^d$ with height restricted by an integer $h$ (that will be allowed to grow with $n$). We assume that the edge weight distribution $F$ satisfies a standard admissibility criterion, defined below.

\begin{defn}\label{def:admble}
	Given the dimension $d$, we call a probability distribution function $F$ on the real line \emph{admissible} if $F$ is supported on $[0,\infty)$, is nondegenerate and we have $F(\gl)<p_c(d)$ where $\gl$ is the smallest point in the support of $F$ and  $p_c(d)$ is the critical probability for Bernoulli bond percolation in~$\dZ^d$.     
\end{defn} 
       
For simplicity we will consider only first-passage time from $\mvzero$ to $n\mve_1$ where $\mve_1$ is the first coordinate vector. The same method can be used to prove similar results for $a(\mvzero,n\mvx)$ where $\mvx$ has rational coordinates. Define  $a_n(h)$ as the  first-passage time  to the point $n\mve_1$ from the origin in the graph $\dZ\times [-h,h]^{d-1}$, formally
\[
 a_n(h):=\inf\{\go(\cP)\mid \cP \text{ is a path from $\mvzero$ to $n\mve_1$ in } \dZ\times [-h,h]^{d-1}\}.
\]   
Here, by $[-h,h]$ we mean the subset $[-h,h]\cap \dZ$ of $\dZ$.  
Informally, $a_n(h)$ is the minimal passage time over all paths which deviate from the straight line path joining the two end points by a distance at most $h$.
We also consider cylinder first-passage time (see Smyth and Wierman~\cite{sw78}, Grimmett and Kesten~\cite{gk84}). A path $\cP$ from $\mvzero$ to $n\mve_1$ is called a cylinder path if it is contained within the $x_1=0$ and $x_1=n$ planes. We define 
\begin{align*}
t_n(h)&:=\inf\{\go(\cP)\mid \cP \text{ is a path from $\mvzero$ to $n\mve_{1}$ in } [0,n]\times [-h,h]^{d-1}\}
\text{ and}\\
 T_n(h) & :=\inf\{\go(\cP)\mid \cP \text{ is a path connecting  } \{0\}\times  [-h,h]^{d-1} \text{ and }\\
   & \qquad\qquad\{n\}\times [-h,h]^{d-1} \text{ in } [0,n]\times [-h,h]^{d-1}\}.
\end{align*}
Clearly $a_n(h), t_{n}(h)$ and  $T_n(h)$ are non-increasing in $h$ for any $n\ge 1$. 
Our main result is that for cylinders that are `thin' enough, we have Gaussian CLTs for $a_n(h), t_n(h)$ and $T_n(h)$ after proper centering and scaling. 

\begin{theorem}\label{thm:mainfpp1}
	    Suppose that the edge-weights $\go_e$'s are i.i.d.~ random variables from an admissible distribution $F$. Suppose $\E[\go^p]<\infty$ for some $p>2$. 
	Let $\{h_n\}_{n\ge 1}$ be a sequence of integers satisfying $h_n=o(n^\ga)$ where  
\begin{align*}
	 \ga< 
        \frac{1}{d+1 + {2(d-1)}/{( p - 2)}}   
\end{align*}     
Then we have
	\[
		\frac{a_n(h_n) - \E[a_n(h_n)]}{\sqrt{\var(a_n(h_n))}} \weakc N(0,1) \text{ as } n\to\infty.
	\]                                                
In particular, if $\E[\go^p]<\infty$ for all $p\ge 1$ then the CLT holds when $h_n=o(n^\ga)$ with $\ga<1/(d+1)$.
If $h_n=O(1)$ then the $F(\gl)<p_c(d)$ condition is not needed. Moreover, the same result is true for $t_n(h_n)$ and $T_n(h_n)$.  
\end{theorem} 
In Section \ref{sec:main}, we will present a generalization of this result (Theorem~\ref{thm:maingr}) to cylinders of the form $\dZ \times G_n$ where $\{G_n\}$ is an arbitrary sequence of undirected connected graphs.


Theorem~\ref{thm:mainfpp1} give rise to a new exponent $\change{\gc_{F}(d)}$ defined as 
\begin{align*}
\change{\gc_{F}}(d)&:=\sup\biggl\{\ga : \frac{a_n(n^\ga) - \E[a_n(n^\ga)]}{\sqrt{\var(a_n(n^\ga))}} \weakc N(0,1) \text{ as } n\to\infty\biggr\}.	
\end{align*}                           
Clearly we have $\change{\gc_{F}(d)}\ge 1/(d+1)$ for $F$ having all moments finite and satisfying the conditions in Theorem \ref{thm:mainfpp1}. 

Is $\change{\gc_{F}(d)}$  actually equal to $1/(d+1)$? We do not have a rigorous answer for that yet. However, under some well known but unproven hypotheses about existence of fluctuation exponent $\chi(d)$ and transversal exponent $\xi(d)$ (see the next Subsection~\ref{sec:fluc}), and strict convexity of the limiting shape we prove in Sections~\ref{sec:clt_ht} and \ref{sec:pclt} that 
$\change{\gc_{F}(d)}=\xi(d)$ 
when the fluctuation exponent is strictly positive, or if the dimension is $2$ or $3$. For $d=2$, this result is also supported by numerical simulations (Section \ref{sec:num}).

\begin{conj}[Partly proved in Sections \ref{sec:clt_ht} and \ref{sec:pclt}]\label{conj:onethird}
For all $d\ge 2$ and $F$ having finite exponential moment, $\change{\gc_{F}(d)} = \xi(d)$. 
\end{conj}

An interesting feature of the proof of Theorem \ref{thm:mainfpp1} is that while it is relatively easy to get a CLT for cylinders of width $n^{\alpha}$ for $\alpha$ sufficiently small, to go all the way up to $\alpha = 1/(d+1)$ one needs a somewhat complicated `renormalization' argument that has to be taken to a certain depth of recursion, where the depth depends on how close $\alpha$ is to $1/(d+1)$. This renormalization step is required because of the gap in the lower and upper bounds for the moments. We believe this step can be removed with the correct order for the moments.

A deficiency of Theorem \ref{thm:mainfpp1} is that we do not have formulas for the mean and the variance of $a_n(h_n)$. Still, we have some bounds: the following result states that under the hypotheses of Theorem~\ref{thm:mainfpp1} the mean grows linearly with $n$ and the growth rate does not depend on $h_n$  as long as $h_n\to\infty$. It also gives upper and lower bounds for the variance of $a_n(h_n)$.

\begin{prop}\label{prop:mvbd}
	Let $\mu_n(h_n)$ and $\gs_n^2(h_n)$ be the mean and variance of $a_n(h_n)$. Assume that $h_n \to \infty$ as $n\to\infty$. Then  
	\[
	  \lim_{n\to \infty} \frac{\mu_n(h_n)}{n}= \nu(\mve_1),
	\]                                        
	where $\nu(\mve_1)$ is defined as in \eqref{hw}. Moreover, if $F$ is admissible  we have
	\[
		c_1 \frac{n}{h_n^{d-1}} \le \gs_n^2(h_n) \le c_2 n  
	\]                                     
	for some absolute constants $c_1,c_2>0$ depending only on $d$ and $F$. If $h_n=h$ for all $n$ for fixed $h\in (0,\infty)$, then both $\lim_{n\to \infty} \mu_n(h)/n$ and $\lim_{n\rightarrow\infty} \sigma_n^2(h)/n$ exist and are positive for any non-degenerate distribution $F$ on $[0,\infty)$, but their values depend on $h$.
\end{prop}



In fact when $h_n=h$  for all $n$ for fixed $h\in(0,\infty)$, we can say much more.   
Define $\mu(h):=\lim_{n\to\infty} \mu_n(h)/n $ and $\gs^2(h):=\lim_{n\to\infty}\gs^{2}_n(h)/n$. Existence of the limits follow from Proposition~\ref{prop:mvbd}. Now consider the 
continuous process $X(\cdot)$ defined by  $X(n)=t_n(h)-n\mu(h)$ for $n\in \{0,1,\ldots\}$ and extended by linear interpolation. Then we have the following result.

\begin{prop}\label{prop:bmh}
	Assume that  $\E[\go^p]<\infty$ for some $p>2$ where $\go\sim F$. Then the scaled process $\{ (n\gs^2(h))^{-1/2} X(nt) \}_{ t\ge 0}$ converges in distribution to the standard Brownian motion as $n\to\infty$. 
\end{prop}

Here we mention that even though we have lower and upper bounds for the variance of $a_{n}(h_{n})$ in Proposition~\ref{prop:varbd}, none of the bounds seem to be the correct one, at least when $d=2$ as $h_{n}\to\infty$. In fact numerical simulation results suggests the following.

\begin{conj}\label{conj:varbd}
For $d=2$ and $h_{n}\ll n^{2/3}$, $\var(a_{n}(h_{n}))=\Theta(nh_{n}^{-1/2})$.
\end{conj}

Finally let us mention that a variant of Theorem~\ref{thm:mainfpp1} can be proved for the undirected first-passage  {\it site} percolation model also. Here instead of edge-weights $\{\go_e\mid e\in E\}$ we have vertex weights $\{\go_{\mvx} \mid \mvx\in \dZ^d\}$ and travel time  for a path $\cP$ is defined by $\go(\cP)=\sum_{v\in \cP} \go_v$. The same proof technique should work. The same remark also holds for semi-directed first-passage model where the paths are not allowed to move backward in a particular direction. 


\subsection{Fluctuation exponents}\label{sec:fluc}
In the physics literature, there are two main exponents $\chi$ and $\xi$ that describe, respectively, the longitudinal and transversal fluctuations of the growing surface $B_t$. For example, it is expected under mild conditions that the first-passage time $a(\mvzero,n\mvx)$ has standard deviation of order $n^\chi$, and the exponent $\chi$ is independent of the direction $\mvx\in\dZ^d$. It is also expected that all the paths achieving the minimal time   $a(\mvzero,n\mvx)$ deviate from the straight line path joining $\mvzero$ to $n\mvx$ by distance at most of the order of $n^\xi$, that is all the minimal paths are expected to lie entirely inside the cylinder centered on the straight line joining $\mvzero$ to $n\mvx$ whose width is of the order of $n^\xi$ (see Section \ref{sec:clt_ht} for a rigorous definition of $\chi$ and $\xi$). 

In general the exponents $\chi$ and $\xi$ are expected to depend only on the dimension $d$ not the distribution $F$. Moreover they are also conjectured to satisfy the scaling relation $\chi=2\xi-1$ for all $d$ (see Krug and Spohn~\cite{ks91}). \change{Very recently this relation has been proved under certain natural but unproven assumptions (see Chatterjee~\cite{chatterjee11}, and Auffinger and Damron~\cite{am11}).}
The predicted values for $d=2$ (for models whose exponents are expected to be same in all directions) are $\chi=1/3$ and $\xi=2/3$ (see Kardar, Parisi and Zhang~\cite{kpz86}). For higher dimensions there are many conflicting predictions. However it is believed that above some finite critical dimension $d_c$, the exponents satisfy $\chi=0$ and $\xi=1/2$. 

We briefly describe the rigorous results known about the exponents $\chi$ and $\xi$. The first nontrivial upper bound on the variance of $a(\mvzero,n\mvx)$ was $O(n)$ for all $d$ due to Kesten~\cite{kes93}.  The best known upper bound of $n/\log n$ is due to Benjamini, Kalai and Schramm \cite{bks03}.
In $d=2$ the best known lower bound of $\log n$  is due to Pemantle and Peres \cite{pp94} for exponential edge weights, 
Newman and Piza \cite{np95} for general edge weights satisfying $F(0)<p_c(2)$ or $F(\gl)<p_c^{dir}(2)$ for $\gl$ being the smallest point in the support of $F$ where $p_c^{dir}(2)$ is the critical probability for directed Bernoulli bond percolation, 
and Zhang~\cite{zhang08} for $\mvx=\mve_1$ and edge weight distributions having finite exponential moments and satisfying 
$F(\gl)\ge p_c^{dir}(2), F(\gl-)=0,\gl>0$.  

Hence the only nontrivial bound known for $\chi$ is $\chi\le 1/2$. Note that the bound $0\le \chi\le 1/2$ along with the scaling relation (which is unproven) would imply that $1/2\le \xi\le 3/4$. In fact using a closely related exponent $\chi'$ which satisfies $\chi'\ge 2\xi-1$ and $\chi'\le 1/2$ (see Newman and Piza~\cite{np95}, Kesten~\cite{kes93} and Alexander~\cites{alex93,alex97}), it was proved in \cite{np95} that $\xi\le 3/4$ in any dimension for paths in the directions of strict convexity of the limit shape.  Moreover, Licea, Newman and Piza~\cite{lnp96}, comparing appropriate variance bounds, proved that $\xi(d)\ge 1/(d+1)$ for all dimensions $d$. They also proved that $\xi'(d)\ge 1/2$ for all dimensions $d$ for a related exponent $\xi'$ of $\xi$.

Our results show that under some natural but unproven assumptions, $\xi(d)$ is also expected to be the threshold where the Gaussian CLT breaks down.


\subsection{Comparison with directed last-passage percolation} 
\label{sub:lit}
In all the previous discussions we used undirected first-passage times.
A directed model is obtained when instead of all paths, one considers only directed paths. A directed path is a path that moves only in the positive direction at each step (e.g.\ in $d=2$, the path moves only up and right).  Let us restrict ourselves to $d=2$ henceforth. 
The  directed (site/bond) last-passage time to the point $(n,h)$ starting from the origin is defined as 
\begin{align*}
   L^s_\uparrow(n,h)&:= \sup\{ \go(\cP)\mid \cP \in \Pi(n,h)\},
\end{align*}
where $\Pi(n,h)$ is the set of all directed paths from $(0,0)$ to $(n,h)$. Note that all the paths in $\Pi(n,h)$ are inside the rectangle $[0,n]\times [0,h]$. 

The directed last-passage site percolation model in $d=2$ has received particular attention in recent years, due to its myriad connections with the totally asymmetric simple exclusion process, queuing theory and random matrix theory. An important breakthrough, due to Johansson~\cite{kj00}, says that when the vertex weights $\go_{\mvx}$'s are i.i.d.~geometric random variables,   $L^s_\uparrow(n, n)$  has fluctuations of order $n^{1/3}$ and has the same limiting distribution as the largest eigenvalue of a GUE random matrix upon proper centering and scaling. (This is also known as the Tracy-Widom law.) Moreover, this holds if we replace $L^s_\uparrow(n,n)$ with $L^s_\uparrow(n,\lfloor \gr n\rfloor)$ for any $\rho \in (0,1]$. This continues to hold if one replaces geometric by exponential or bernoulli random variables~\cites{kj01,gtw01}. \change{However universality of this limit result for a general class of vertex weight distributions  is still open. }

Since the above result holds for arbitrary $\rho > 0$, one can speculate whether we can actually take $\rho \rightarrow 0$ as $n \rightarrow \infty$, i.e.\ look at directed last-passage percolation in thin rectangles. Indeed, the analog of Johansson's result in this setting was proved by several authors~\cites{bs05,bm05,suidan06}  in recent years for quite a general class of vertex weight distributions, provided the rectangles are `thin' enough \change{(in particular for $\rho=n^{-(1-\ga)}$ with $\ga<3/7$ when the vertex weights have finite moments of all order)}. This contrasts starkly with our result about the Gaussian behavior of first-passage percolation in thin rectangles.

\subsection{Structure of the paper} 
\label{sub:structure_of_article}
The article is organized \change{as follows}. In Section \ref{sec:main} we state a general result that encompasses Theorem~\ref{thm:mainfpp1}. In Section \ref{sec:mean} we prove the asymptotic behavior of the mean of $a_n(G_n)$. Sections~\ref{sec:varlbd} and \ref{sec:pfmbd} contain, respectively, the lower bound for the variance and upper bounds for general central moments of $a_n(G_n)$.   In Section~\ref{sec:maingr} we prove the generalized version of Theorem~\ref{thm:mainfpp1}.  We consider the case of first-passage time across $[0,n]\times G$ when $G$ is a fixed graph in Section~\ref{sec:fixedg}. All the results till Section~\ref{sec:fixedg} are unconditional. However, when $G_{n}=[-h_{n},h_{n}]^{d-1}$ one can prove the CLT for a wider range of $h_{n}$ under a few natural but unproved assumptions. In Section~\ref{sec:clt_ht} we state the CLT under the assumption of  existence of fluctuation and transversal exponents, positive curvature of the limiting shape, etc. and we prove the stated results in  Section \ref{sec:pclt}. In particular, we show that the CLT holds all the way upto the height of the unrestricted geodesic under those assumptions in dimension $2$ and $3$. Finally in Section~\ref{sec:num} we present the numerical results. 

             
\section{Generalization} 
\label{sec:main}
In this section, we generalize the theorems of Section \ref{sec:introduction} to first-passage percolation on graphs on the form $\dZ\times G_n$, where $\{G_n\}$ is an arbitrary increasing sequence of \change{finite} undirected graphs with $k_n$ many edges and having diameter $d_n$. \change{To get the results for $d$-dimensional square lattice one takes $G_{n}=[-h_n,h_n]^{d-1}$ for $n\ge 1$.}

Before stating the results, let us fix our notations. The set $\{a,a+1,\ldots,b\}$ with the nearest neighbor  graph structure will be denoted by $[a,b]$. When $a=0$, we will simply write $[b]$ instead of $[0,b]$. Throughout the rest of the article we will consider the undirected first-passage bond percolation model with edge weight distribution $F$, as defined in the previous section. Let $\mu$ and $\gs^2$ be  the mean and the variance of $F$. 
We will use the standard notations $a_n=O(b_n)$ and $a_{n}=o(b_{n})$, respectively, in the case $\sup_{n\ge 1} a_{n}/b_{n}<\infty$ and $\lim_{n\to\infty}a_{n}/b_{n}=0$.

For two  finite connected graphs  $H$ and $G$, we define the product graph structure on $H\times G$ in the natural way, that is, there is an edge between $(u,w)$ and $(v,z)$ if and only if either $(u,v)$ is an edge in $H$ and $w = z$, or $u=v$ and $(w,z)$ is an edge in $G$.

We will consider first-passage percolation on a special class of product graphs. Fix an integer $n$ and a connected graph $G$ with a distinguished vertex $o\in G$. Let  $a_{n}(G)$ denote  the first-passage time from $(0,o)$ to $(n,o)$ in $\dZ\times G$. That is, 
\begin{align*}
	a_n(G) := \inf\{\go(\cP)\mid \cP \text{ is a path from } (0,o) \text{ to } (n,o) \text{ in } \dZ\times G\}
\end{align*}                    
where $\go(\cP):=\sum_{e\in\cP}\go_e$ is weight of the path $\cP$. We define the cylinder first-passage time $t_n(G)$ as 
\begin{align*}
	  t_n(G):=\inf\{\go(\cP):\cP \text{ is a path from $(0,o)$ to $(n,o)$ in } [0,n]\times G\}. 
\end{align*}
We also define the side-to-side (cylinder) first-passage time as follows: 
\begin{align} 
	\begin{split} 
		T_{a,b}(G):=\min&\{\go(\cP)\mid \cP \text{ is a path connecting the two sides }\\
		&\qquad \{a\}\times G \text{ and $\{b\}\times G$ in } [a,b]\times G\},
	\end{split}  \label{def:TabG} 
\end{align}                          
that is, $T_{a,b}(G)$ is the minimum weight among all paths that join the right boundary of the product graph $[a,b]\times G$ to the left boundary of it.  Note that it is enough to consider only those paths that start from some vertex in $\{a\}\times G$  and end at some vertex in $\{b\}\times G$, and lie in the set $[a+1,b-1]\times G$ throughout except for the first and last edges. One implication of this fact is that $T_{a,b}(G)$ is independent of the weights of the edges in the left and right boundaries $\{a\}\times G,\{b\}\times G$. We will write $T_{0,n}(G)$ simply as $T_n(G)$.

Now consider a nondecreasing sequence of connected graphs  $G_n=(V_n,E_n)$, $n\ge 1$.  By `nondecreasing' we mean that $G_n$ is a subgraph (need not be induced) of $G_{n+1}$ for all $n$. Let $o$ be a distinguished vertex in $G_1$, which we will call the \emph{origin} of $G_1$. Then $o\in G_n$ for all $n$. Let $k_n$ and $d_n$ be the  number of edges and the diameter of $G_n$, respectively.

Our object of study is first-passage percolation on  the product graph $\dZ\times G_n$ with i.i.d.\ edge weights from the distribution $F$.  In particular, we wish to understand the behavior of the first-passage time $a_n(G_n)$ from $(0,o)$ to $(n,o)$.  

The main result of this section is the following.
\begin{theorem}\label{thm:maingr}
	Let $G_n$ be a nondecreasing sequence of connected graphs with a fixed origin $o$. Let $d_{n}$ and $k_{n}$ be the diameter and the number of edges in $G_{n}$. Suppose that as $n \rightarrow \infty$, $k_n=O(d_n^\theta)$ for some fixed $\theta\ge 1$. Let $a_n(G_n)$ be the first-passage percolation time from $(0,o)$ to $(n,o)$ in the graph $\dZ\times G_n$. Suppose that a generic edge weight $\go$ satisfies $\E[\go^p]<\infty$ for some $p>2$. Then we have
	\[
		\frac{a_n(G_n)-\E[a_n(G_n)]}{\sqrt{\var(a_n(G_n))}} \weakc N(0,1)
	\]                                                     
	as $n\to\infty$ provided one of following holds:
	\begin{enumerateA}
		\item  There is a fixed connected graph $G$ such that  $G_n=G$ for all $n\ge 1$, or   
		\item  $G_n$'s are connected subgraphs of $\dZ^{d-1}$ for some $d> 1$, the edge weight distribution is admissible and $d_n=o(n^\ga)$, where 
		\[
		\ga<\frac{1}{2+\theta+{2\theta}/({ p-2})}.
		\]                                                    
	\end{enumerateA} 
	Moreover, the same result holds for $t_n(G_n),T_{n}(G_n)$ in place of $a_n(G_n)$.
\end{theorem}  

Clearly, this theorem implies Theorem \ref{thm:mainfpp1} by taking $G_n=[-h_n,h_n]^{d-1}$ with $d_n=2h_n(d-1)^{1/2}$ and $\theta=d-1$. Throughout the rest of the paper we will consider the case of general sequence $G_n$. 

As we remarked earlier we do not have explicit formulas for the mean and the variance of $a_n(G_n)$. The following result is the generalization of the `mean part' of  Proposition \ref{prop:mvbd}. 
\begin{proposition}\label{prop:meanbd}
Consider the setup introduced above. Then the limit
	$$
		\nu:=\lim_{n\to\infty} \frac{1}{n}\E[a_n(G_n)]
	$$                          
	exists and we have
	\[
		\nu n \le \E[a_n(G_n)] \le \mu n    \text{ for all } n.
	\] 
Moreover, $\nu>0$ if $G_n=G$ for all $n\ge 1$ or $G_n$'s are subgraphs of $\dZ^{d-1}$ and $F(0)<p_c(d)$.  In particular, when $G_n=[-h_n,h_n]^{d-1}$
and $h_n\to\infty$ as $n\to\infty$, we have $\nu = \nu(\mve_1)$, where $\nu(\mve_1)$ is defined as in \eqref{hw}. 
We also have 
\[
	 \E[a_n(G_n)]\le \E[t_n(G_n)] \le \E[T_{n}(G_n)] + 2\mu d_n \le \E[a_n(G_n)] + 2\mu d_n
\]                                                                 
for all $n$.
\end{proposition}
 
Now let us state the upper and lower bounds for the variance of $a_n(G_n)$, i.e.\ the `variance part' of Proposition \ref{prop:mvbd}.
\begin{prop}\label{prop:varbd}
	 Under the condition of Theorem~\ref{thm:maingr} we have
	\begin{align*}
	   c_1 \frac{n}{k_n} \le \var(a_n(G_n)) \le c_2 n  
	\end{align*}
	for some positive constants $c_1$, $c_2$ that do not depend on $n$. Moreover, $\lim_{n\to\infty}\var(a_n(G_n))/n$ exists for all non-degenerate distribution $F$ on $[0,\infty)$ when $G_n=G$ for all $n$. The above results hold for $t_n(G_n)$ and $T_{n}(G_n)$.  
\end{prop}     
     
In fact when $G_n=G$ for all $n\ge 1$, we can say much more as in Proposition~\ref{prop:bmh}.   
Define 
\begin{align}
	   \mu(G)&:=\lim_{n\to\infty} \frac{\E[a_n(G)]}{n} 
	\text{ and } \gs^2(G):=\lim_{n\to\infty}\frac{\var(a_n(G))}{n}.
\end{align} 
Existence and positivity of the limits follow from Propositions~\ref{prop:meanbd} and \ref{prop:varbd}. Consider the 
continuous process $X(\cdot)$ defined by  $X(n)=t_n(G)-n\mu(G)$ for $n\ge 0$ and extended by linear interpolation. Then we have the following result. 
\begin{prop}\label{prop:bm}
	Assume that the generic edge weight $\go$ is non-degenerate and satisfies  $\E[\go^p]<\infty$ for some $p>2$. Then the scaled process $$\{ (n\gs^2(G))^{-1/2} X(nt) \}_{ t\ge 0}$$ converges in distribution to the standard Brownian motion as $n\to\infty$. 
\end{prop}     

%
%

\section{Estimates for the mean} 
\label{sec:mean}

In this section we will prove Proposition~\ref{prop:meanbd}. We will break the proof into several lemmas. Lemma~\ref{lem:p2s} shows that the random variables $a_n(G_n),t_n(G_n)$ and $T_{n}(G_n)$ are close in $L^p$ norm when the diameter $d_n$ of $G_n$ is small.  Note that the maximum weight over all self avoiding paths in $G_{n}$ is of the order of $d_{n}$.                             

\begin{lemma}\label{lem:p2s} 
	We have  
	\[
	T_{n}(G_n)\le a_n(G_n) \le t_{n}(G_n) \text{ for all } n.
	\]
	Moreover we have 
	\begin{align*}
	    \E[|t_n(G_n)-T_{n}(G_n)|^p]\le 2^pd_n^p \E[\go^p] \text{ for all } n\ge 1
	\end{align*}  
	when $\E[\go^p]<\infty$ for some $p\ge 1$ and a typical edge weight $\go\sim F$.
\end{lemma}
\begin{proof} 
    Fix any path $\cP$ from $(0,o)$ to $(n,o)$ in $\dZ\times G_n$. The path $\cP$ will hit $\{0\}\times G_n$ and $\{n\}\times G_n$ at some vertices. Let $(0,u)$ be the vertex where $\cP$ hits $\{0\}\times G_n$ the last time and $(n,v)$ be the vertex where $\cP$ hits $\{n\}\times G_n$ the first time after hitting $(0,u)$. The path segment of $\cP$ from $(0,u)$ to $(n,v)$ lies inside $[n]\times G_n$ and by non-negativity of edge weights we have $\go(\cP)\ge T_{n}(G_n)$. Since this is true for any path $\cP$ joining $(0,o)$ to $(n,o)$ in $\dZ\times G_n$, we have $T_{n}(G_n)\le a_n(G_n)$.   

Clearly $a_n(G_n)\le t_{n}(G_n)$. Combining the two inequalities, we see that 
\[
T_{n}(G_n)\le a_n(G_n) \le t_{n}(G_n) \text{ for all } n.
\]
Since the number of paths joining the left side $\{0\}\times G_n$ to the right side $\{n\}\times G_n$ in $[0,n] \times G_n$ is finite there is a path achieving the minimal weight $T_{n}(G_n)$. Choose such a path $\cP^*$ using a deterministic rule. Suppose that the path $\cP^*$ starts at $(0,u)$ and ends at $(n,w)$. As we remarked earlier in Section \ref{sec:main} the random variables $T_{n}(G_n),\cP^*,u,w$ are independent of the edge weights $\go_e$ where $e$ is an edge in $\{0\}\times G_n$ or $\{n\}\times G_n$. 

Let $\cP(u),\cP(w)$ be some minimal length paths in $G_n$ joining $o,u$ and $o,w$ respectively. We have
$
 t_{n}(G_n) - T_{n}(G_n) \le S_n
$
where $S_n$ is the sum of edge weights in the paths $\{0\}\times\cP(u)$ and  $\{n\}\times\cP(w)$ and hence
\[
\E[|t_n(G_n)-T_{n}(G_n)|^p]\le \E[S_n^p].
\]   
Moreover by independence of $u,w$ and the edge weights in $\{0,n\}\times G_n$ we have $\E[S_n^p|u,w]\le (|\cP(u)|+|\cP(w)|)^p\E[\go^p]$. By definition of diameter we have $|\cP(u)|+|\cP(w)|\le 2d_n$ and thus we are done.         
\end{proof}
 
The following lemma combined with Lemma~\ref{lem:p2s} completes half of the proof of Proposition~\ref{prop:meanbd}. Recall that $\{G_n\}$ is a nondecreasing sequence of finite connected graphs.

\begin{lemma}\label{lem:meanbd}
	   The limit
\[
		\nu=\lim_{n\to\infty} \frac{\E[a_n(G_n)]}{n}
\]                         
exists and we have
\[
		\nu n  \le \E[a_n(G_n)] \le \mu n \text{ for all } n.
\] 
Moreover, we have $\nu<\mu$ if $d_n\ge 1$ and $F$ is non-degenerate.
\end{lemma} 
\begin{proof}
		Considering the straight line path from $(0,o)$ to $(n,o)$ it is easy to see that $\E[a_{n}(G_n)]\le \mu n$.                           
The existence of the limit is easily obtained from subadditivity as follows. Fix $n,m$. Consider $G_n$ and $ G_m$ as subgraphs of  $G_{n+m}$. Let  $a_{n,n+m}(G_m)$ denote the  first-passage time in $\dZ\times G_m$ from $(n,o)$ to $(n+m,o)$. Clearly $a_{n,n+m}(G_m)\equald a_m(G_m)$. Joining the minimal weight paths from $(0,o)$ to $(n,o)$ achieving the weight $a_n(G_n)$ and from $(n,o)$ to $(n+m,o)$ achieving the weight $a_{n,n+m}(G_m)$, we get a path in $\dZ\times G_{n+m}$ from $(0,o)$ to $(n+m,o)$. Clearly
	\begin{align*}
		a_{n+m}(G_{n+m}) &\le a_{n}(G_n)  + a_{n,n+m}(G_m).
	\end{align*} 
	Now taking expectation in both sides and using the subadditive lemma we have
	\[
		\nu  := \lim_{n\to\infty}\frac{\E[a_n(G_n)]}{n}
	\]  
	exists and equals $\inf_{n\ge 1}{\E[a_n(G_n)]}/{n}$.

	To show that $\nu<\mu$ it is enough to consider the one edge graph $G_n=G=\{0,1\}$ and $n$ even. Consider the following two paths from $(0,0)$ to $(2n,0)$. One is the straight line path. The other is the path connecting $(0,0),(0,1),(1,1),(1,0),(2,0)$ and repeating the same pattern. Clearly we have $\E[a_{2n}(G)]\le \mu n + n\E[\min\{\go_1,\go_2+\go_3+\go_4\}]$ where $\go_i$'s are i.i.d. from~$F$. From here it is easy to see that $\nu<\mu$.   
\end{proof}

We complete the proof of Proposition~\ref{prop:meanbd} by finding lower bound for $\nu$ under appropriate conditions. Recall that $\nu(\mve_1)>0$ iff $F(0)<p_c(d)$ where $\mve_1$ is the first coordinate vector in $\dZ^d$ and $\nu(\mvx)$ is defined as in  \eqref{hw}.

\begin{lemma}\label{lem:posmean} 
	Suppose $G_n$'s are subgraphs of $\dZ^{d-1}$. Then the limit $\nu$ in Lemma~\ref{lem:meanbd} satisfies
	\[
		\nu\ge \nu(\mve_1)
	\]                   
	where $\nu(\mve_1)$ is as defined in \eqref{hw}.                
	Equality holds when $G_n=[-h_n,h_n]^{d-1}$ with  $ h_n\to\infty$ as $n\to\infty$.   
	Moreover, the limit $\nu$ is positive if $G_n=G$ for all~$n$. 
\end{lemma}
\begin{proof}
First suppose that $G_n=G$ for all $n$ and $G$ has $v$ vertices. It is easy to see that $\E[a_{n}(G_n)]\ge n\E[Y]$ where $Y$ is the minimum of $v$ i.i.d.~random variables each having distribution $F$, because any path from $(0,o)$ to $(n,o)$ must contain at least one edge of the form $((k,u), (k+1,u))$ for each $k=0,\ldots, n-1$. Since $\E[Y]>0$, it follows that $\nu > 0$.                          

	Now consider the case when $G_n$'s are subgraphs of $\dZ^{d-1}$ (we will match $o$ with the origin in $\dZ^{d-1}$). Then $\dZ\times G_{n}$ is a subgraph of $\dZ^{d}$ with $(0,o)=\mvzero$ and $(n,o)=n\mve_{1}$ where $\mvzero$ and $\mve_1$ denote the origin and the first coordinate vector in $\dZ^{d}$. Clearly we have $a(\mvzero,n\mve_1)\le a_{n}(G_n)$ for all $n$. Diving both sides by $n$ and taking expectations  we have 
	\[
		\nu= \lim_{n\to\infty}\frac{1}{n}\E[a_n(G_n)]\ge \lim_{n\to\infty}\frac{1}{n}\E[a(\mvzero,n\mve_1)]=\nu(\mve_1).
	\]                                             
	To prove that $\nu = \nu(\mve_1)$ when $G_n=[-h_n,h_n]^{d-1}$, break the cylinder graph $[n]\times G_n$ into smaller cylinder graphs of length $\lfloor l_n/C\rfloor$ for some fixed constant $C>0$ where $l_n=\min\{n^{1/2},h_n\}$. Note that concatenating paths from $(il_n/C,o)$ to $((i+1)l_n/C,o)$ for $i= 0,1,\ldots$ we get a path from $(0,o)$ to $(n,o)$. Let $n=m\lceil l_n/C\rceil + r$ with $r<\lceil l_n/C\rceil$. Thus we have
	\begin{align}
		\E[a_n(G_n)]\le m \E[X(\lceil l_n/C\rceil,l_n)] + \E[X(r,l_n)]\label{eq:xmn}
	\end{align}                                                  
	where 
	\begin{align*}
		X(n,h)&:=\inf\{\go(\cP)\mid \cP \text{ is a path from $(0,o)$ to $(n,o)$}  \text{ that lies in the }\\
		&\text{ rectangle } [1,n-1]\times[-h,h]^{d-1} \text{ except for the first and last edge} \}.
	\end{align*} 
	Dividing both sides of \eqref{eq:xmn} by $n$ and taking limits (note $l_n=o(n)$ and $l_n\to\infty$ as $n\to\infty$) we have
	\begin{align*}
			\nu:=\lim_{n\to\infty}\frac{1}{n}\E[a_n(G_n)] &\le \liminf_{n\to\infty} \frac{\E[X(\lceil n/C\rceil,n)]}{\lceil n/C\rceil}
			\le \lim_{n\to\infty} \frac{\E[X(n, \lfloor Cn\rfloor)]}{n} 
	\end{align*}        
	for any $C>0$. The last limit exists by subadditivity. Denote the last limit by $\ga(C)$ which also satisfies $\ga(C)=\inf_{n}\E[X(n, \lfloor Cn\rfloor)/n$. Now let us consider the unrestricted cylinder percolation time 
$t(\mvzero,n\mve_1)$ defined as the minimum weight among all paths from $\mvzero$ to $n\mve_{1}$ lying in the vertical strip  $0< x_1 < n$  except for the first and the last edge. From standard results in first-passage percolation theory (see Section~$5.1$ in Smythe and Wierman~\cite{sw78} for a proof) we have 
	\[
	   \lim_{n\to\infty}\frac{1}{n}\E[t(\mvzero,n\mve_1)]=\nu(\mve_{1}). 
	\]                                      
	Now for fixed $n$, the random variables $X(n,\lfloor Cn\rfloor)$ are decreasing in $C$ and $t(\mvzero,n\mve_1)=\lim_{C\to\infty}X(n,\lfloor Cn\rfloor)$. By monotone convergence theorem we have 
	\[
	\E[t(\mvzero,n\mve_1)]=\lim_{C\to\infty}\E[X(n,\lfloor Cn\rfloor)]\ge \limsup_{C\to\infty}\ga(C)n\ge \nu n.
	\] 
	Dividing both sides by $n$ and letting $n\to\infty$ we are done.  
\end{proof}   


\section{Lower bound for the variance} 
\label{sec:varlbd}
 
Here we will prove the lower bound for the variance given in Proposition~\ref{prop:varbd}. First we will prove a uniform lower bound that holds for any $n$ and $G$.  Later we will specialize to the case $G=G_n$ for given $n$.

\begin{lemma}\label{lem:lbd}
	 Let  $G$ be a  subgraph of $\dZ^{d-1}$ with diameter $D$ and number of  edges $k$. Let $F$ be admissible. Then we have
	\begin{align}
	   \var(t_n(G)) \ge c_1\frac{n}{k} \text{ and } \var(T_{n}(G)) \ge c_1\frac{n}{k}\left( 1 - c_2\frac{D}{n}\right) 
	\end{align}
	for some absolute positive constants $c_1,c_2$ that depend only on $d$ and $F$. The same result holds for all nondegenerate probability distributions $F$ on $[0,\infty)$ with $c_i$ depending only on $G$ and $F$. In particular, when $D\le n/(2c_2)$ we have
\[
	\var(T_{n}(G)) \ge c_3\frac{n}{k}
\]
for all $n,k$ for some absolute constant $c_{3}>0$.
\end{lemma}     

\begin{rem}
The proof of the variance lower bound bears many similarities to the variance bound proofs given in Newman and Piza~\cite{np95} and Benjamini, Kalai, Schramm~\cite{bks03} using influence of random variables. In fact one can view the lower bound as the contribution coming from the first order Fourier terms. When the edge weights are Gaussian (not non-negative) one can give a simpler proof as follows. For any smooth function $f$ of $N$ Gaussian variables one has 
\[
\var(f)=\sum_{k= 1}^{\infty}\sum_{1\le i_{1},i_{2},\ldots,i_{k}\le N} (\E[\partial_{x_{i_1}}\partial_{x_{i_2}}\cdots \partial_{x_{i_k}}f])^{2}
\]
where the $k$-th sum corresponds to the contribution from $k$-th order Fourier coefficients. Using the lower bound for $k=1$ and Cauchy-Schwarz inequality one has
$
\var(f)\ge N^{-1}(\E(\sum_{i=1}^{N}\partial_{x_i}f))^{2}.
$
The same bound holds when $f$ is a Lipschitz function, in particular when $f$ is the minimum path weight function. In that case $\partial_{x_i}f=\ind\{i \text{ is in the optimal path}\}$ and sum over all $i$ gives number of edges in the optimal path. Thus using $N=$ total number of edges $=nk$ and number of edges in the optimal path $\ge n$ we get the variance lower bound $cn/k$. From this heuristic and the fact that for noise sensitive random variables contribution from lower order Fourier coefficients is negligible for the variance, it is also easy to guess why the lower bound is probably not optimal.
\end{rem}


\begin{proof}[Proof of Lemma~\ref{lem:lbd}]
		Fix $G$ and $n$. Let $v$ be the number of vertices in $G$. Let $\{e_1,e_2,\ldots,e_N\}$ be a fixed enumeration of the edges in $[n]\times G$ where $N=(n+1)k+nv$ is the number of edges in that graph. For simplicity let us write  $t_{n}(G)$ simply as $t$. Let $\cF_i$ be the sigma-algebra generated by $\{\go(e_1),\go(e_2),\ldots,\go(e_i)\}$ for $i=0,1,\ldots,N$. For simplicity we will write $\go_i$ instead of $\go(e_i)$. Also we will \change{use $t(\mvgo)$ to explicitly show} the dependence of $t$ on the sequence of edge-weights $\mvgo=(\go_1,\go_2,\ldots,\go_N)$.
	
	Using Doob's martingale decomposition we can write the random variable $t-\E[t]$ as a sum of martingale difference sequences $\E[t| \cF_i]-\E[t| \cF_{i-1}], i=1,2,\ldots,N$. Since martingale difference sequences are uncorrelated we have the standard identity 
	\[
		\var(t)=\sum_{i=1}^N \var( \E[t| \cF_i]-\E[t|\cF_{i-1}] ).
	\]  
	 For $1\le i\le N$, let $\hat{\mvgo}^i$  denote the sequence of edge-weights $\mvgo$ excluding the weight $\go_i$. Moreover, for $x\in\dR^+$, we will write $(\hat{\mvgo}^i,x)$ to denote the sequence of edge-weights where the weight of the edge $e_j$ is $\go_j$ for $j\neq i$ and $x$ for $j=i$. Clearly we have $\mvgo=(\hat{\mvgo}^i,\go_i)$ for $i=1,2,\ldots,N$.  If $\eta$ is a random variable distributed as $F$ and is independent of $\mvgo$, then we have
	$
		\E[t| \cF_i]-\E[t|\cF_{i-1}] = \E[t(\hat{\mvgo}^i,\go_i) - t(\hat{\mvgo}^i,\eta)| \cF_i].
	$    
	It is easy to see that (\change{since $\var(t)\ge \var(\E[t|\cF])$ for any sigma field $\cF$})
	\begin{align*}
		\var(\E[t(\hat{\mvgo}^i,\go_i) - t(\hat{\mvgo}^i,\eta)|\cF_i])
		&\ge \var( \E[\E[t(\hat{\mvgo}^i,\go_i) - t(\hat{\mvgo}^i,\eta)|\cF_i]| \go_i] )  \\
		&=   \var( \E[t({\mvgo})|\go_i]).
	\end{align*}                                                                                          
Now for any random variable $X$  we have
$
	   \var(X)=\frac{1}{2}\E(X_1-X_2)^2 
$                                      
	where $X_1,X_2$ are i.i.d.~copies of $X$. Thus we have	
\begin{align}
	\var( \E[t({\mvgo})|\go_i]) 
	&=\frac{1}{2}\E[(\E[t(\hat{\mvgo}^i,\go_i)-t(\hat{\mvgo}^i,\eta)| \go_i,\eta])^2] \notag
	\\&=\E[(\ind_{\{\go_{i}>\eta\}}\E[t(\hat{\mvgo}^i,\go_i)-t(\hat{\mvgo}^i,\eta)| \go_i,\eta])^2] \label{eq:idvar}
\end{align} 	
where in the last line we have used the fact that $\go_{i}$ and $\eta$ are i.i.d.~.
Define 
\begin{align}
\gD_{i}:= \E[\ind_{\{\go_{i}>\eta\}}(t(\hat{\mvgo}^i,\go_i)-t(\hat{\mvgo}^i,\eta))| \mvgo]
\end{align}
for $i=1,2,\ldots,N$. From \eqref{eq:idvar} we have $\var( \E[t({\mvgo})|\go_i])  \ge (\E[\gD_{i}])^{2}$ for all $i$. Combining we have
\begin{align*}
\var(t) \ge \sum_{i=1}^{N}(\E[\gD_{i}])^{2}\ge \frac{1}{N}\left( \sum_{i=1}^{N}\E[\gD_{i}]\right)^{2} = \frac{1}{N}( \E[g(\mvgo)])^{2}
\end{align*}
where 
\[
g(\mvgo):= \sum_{i=1}^N \gD_{i}=\sum_{i=1}^{N}\E[\ind_{\{\go_{i}>\eta\}}(t(\mvgo)-t(\hat{\mvgo}^i,\eta))| \mvgo] .
\]

Let $\cP_*(\mvgo)$ be a minimum weight path for $\mvgo$ chosen according to a deterministic rule. If the edge $e_i$ is in $\cP_*(\mvgo)$, we have 
\begin{align*}
\ind_{\{\go_{i}>\eta\}}(t(\mvgo)-t(\hat{\mvgo}^i,\eta)) \ge \ind_{\{\go_{i}>\eta\}}(\go_{i}-\eta)=(\go_{i}-\eta)_{+}
\end{align*}
as the weight of the path $\cP_*(\mvgo)$ for the configuration $(\hat{\mvgo}^i,\eta)$ is $t(\mvgo)-\go_{i}+\eta$. Thus we have
\begin{align}
g(\mvgo)\ge \sum_{i : e_i \in\cP_*(\mvgo)} \E[(\go_{i}-\eta)_{+}|\go_{i}].\label{eq:hp}
\end{align}
Now define the function 
\[
h(x)=\E[(x-\eta)_{+}] \text{ where } \eta\sim F.
\]
It is easy to see that $h(x)=0$ iff $x\le \gl$ where $\gl$ is the smallest point in the support of $F$ and $\E[h(\go)]<\infty$. 

Define a new set of edge weights $\go_{i}'=h(\go_{i})$ for $i=1,2,\ldots,N$ with distribution function $F'$. Clearly $\go'_i$'s are i.i.d.~with $F'(0)=\pr(h(\go)=0)=\pr(\go=\gl)$. Moreover let $t(\mvgo')$ be the cylinder first-passage time from $(0,o)$ to $(n,o)$ in $[0,n]\times G$ with edge weights $\mvgo'$. From \eqref{eq:hp} we have $g(\mvgo)\ge t(\mvgo')$. Now from  Lemma~\ref{lem:meanbd} and \ref{lem:posmean} we have $\E[t(\mvgo')] \ge \nu'(\mve_{1})n$ where $\nu'(\mve_{1})$ is as defined in \eqref{hw}  with edge weight distribution $F'$ and $\nu'(\mve_{1})>0$ as $F'(0)<p_{c}(d)$. Also note that $N=(n+1)k+nv\le 3nk$. Thus, finally we have
	\begin{align}
		\frac{1}{n}\var(t)\ge \frac{1}{3k} \left(\frac{\E[t(\mvgo')]}{n}\right)^2 \ge \frac{\nu'(\mve_{1})^2}{3k} \label{eq:vartn} .
	\end{align}

	Now assume that $F$ is any non-degenerate distribution supported on $[0,\infty)$. From Lemma~\ref{lem:posmean} we can see that $\E[t_{n}(G)]\ge cn$ for all $n$ for some constant $c>0$ depending on $G$ and $F$. Thus we are done. 
	
	To prove the result for $T_{n}(G)$ we start with $T_{n}(G)$ in place of $t_{n}(G)$ and use $\E[T_{n}(G)]\ge \E[t_{n}(G)] - 2\mu D$ from Lemma \ref{lem:p2s} in \eqref{eq:vartn}.
\end{proof}
    
\begin{proof}[\bf Proof of the lower bound in Proposition~\ref{prop:varbd}]
From Lemma~\ref{lem:p2s} we have 
	\begin{align*}
		|\var(a_n(G_n))^{1/2} - \var(t_n(G_n))^{1/2}| &\le (\E[|a_n(G_n)-t_n(G_n)|^2])^{1/2}\\& \le 2d_n(\mu^2+\gs^2)^{1/2}
	\end{align*}                                
	for all $n\ge 1$. Now under Theorem~\ref{thm:maingr} we have $d_n=o(n^{1/(2+\theta)})$ which clearly implies that $d_n^2 = o(n/k_n)$ as $k_n=O(d_n^\theta)$. Thus by Lemma \ref{lem:lbd} we are done. Using Lemma~\ref{lem:plen} one can drop the condition $d_n=o(n^{1/(2+\theta)})$ when $F$ is admissible. 
\end{proof}

     
\section{Upper bound for Central moments} 
\label{sec:pfmbd}    
  
In this section we will prove upper bounds for central moments of $a_n(G_n)$, $t_n(G_n)$ and $T_{n}(G_n)$, in particular the upper bound for variance of $a_n(G_n)$ stated in Proposition~\ref{prop:varbd}. Note that by Lemma~\ref{lem:p2s} we have 
\begin{align*}
	\E[|t_{n}(G_n)-a_{n}(G_n)|^p] &\le \E[|t_{n}(G_n)-T_{n}(G_n)|^p]\le \E[(2d_n\go)^p]
\end{align*} 
for all $n$ when $\E[\go^p]<\infty$ for some $p\ge 2$ with $ \go\sim F$. Hence it is enough to prove bounds for $\E[|t_{n}(G_n)-\E[t_{n}(G_n)]|^p]$. 

Fix $n\ge 1$ and a finite connected graph $G$.
We will prove the following.   

\begin{prop}\label{prop:mbd}
	   Let $\E[\go^p]<\infty$ for some $p\ge 2$ and $F(0)<p_c(d)$ where $\go\sim F$. Also suppose that $G$ is a finite subgraph of $\dZ^{d-1}$. Then for any $n\ge 1$ we have
	\[
		\E[|t_{n}(G)- \E[t_{n}(G)] |^p]\le c  n^{p/2}
	\]                                           
	where $c$  is a constant depending only on $p,d$ and $F$. Moreover, the same result holds with $c$ depending on $G$ without any restriction on $F(0)$. The above result holds for $a_n(G)$ and $T_n(G)$ when 
	\[
	D\le Cn^{1/2}
	\]
for some absolute constant $C>0$ where $D$ is the diameter of $G$. 
\end{prop}  

When $F$ has finite exponential moments in some neighborhood of zero, one can use Talagrand's~\cite{tala95} strong concentration inequality along with Kesten's Lemma~\ref{lem:plen} to prove a much stronger result $\pr(|t_{n}(G)- \E[t_{n}(G_n)]|\ge x)\le 4e^{-c_1x^2/n}$ for $x\le c_2n$ for some constants $c_1,c_2>0$. Moreover, one can use moment inequalities due to Boucheron, Bousquet, Lugosi and Massart~\cite{bblm05} to prove that the $p$-th moment is bounded by $n^{p/2}k^{p/2-1}$ for $p\ge 2$. But none of that gives what we need for the proof of Theorem \ref{thm:maingr}, so we have to devise our own proof of Proposition \ref{lem:tech1}.

The next two technical lemmas will be useful in the proof of Proposition~\ref{prop:mbd}. Proofs of  the two technical lemmas and of Proposition~\ref{prop:mbd} are given at the end of this section. 

\begin{lemma}\label{lem:tech1}
	For any $p> 2$ and $x,y\in\dR$ we have
	\[
		\abs{x|x|^{p-2}-y|y|^{p-2}}\le \max\{1,(p-1)/2\}|x-y|(|x|^{p-2}+|y|^{p-2}).
	\]
\end{lemma}    

\begin{lemma}\label{lem:tech2}
	 Let $\gb>1, a,b\ge 0$. Let $y\ge0$ satisfy $y^{\gb}\le a + by$. Then
	\[ 
	  y^{\gb-1}\le a^{(\gb-1)/\gb} + b.
	\]
\end{lemma}

Before proving Proposition~\ref{prop:mbd} we need to define a new random variable $L_{n}(G)$. Consider \change{the }cylinder first-passage time $t_n(G)$ in $[n]\times G$. Call a path $\cP$ from $(0,o)$ to $(n,o)$ in $[n]\times G$ a weight minimizing path if its weight $\go(\cP)$ equals $t_{n}(G)$. An edge $e$ of $[n]\times G$ is called a \emph{pivotal} edge if all weight minimizing paths pass through the edge $e$. Let $L_{n}(G)$ denote the number of pivotal edges given the edge weights $\mvgo$. Clearly $L_{n}(G)$ is a random variable. Lemma \ref{lem:mubd} gives upper bound for the $p$-th central moment of $t_{n}(G)$ in terms of moments of $L_{n}(G)$. Roughly it says that the fluctuation of $t_{n}(G)$ around its mean behaves like square root of $L_{n}(G)$. 

\begin{lemma}\label{lem:mubd}
	     Let $\E[\go^p]<\infty$ for some $p\ge 2$ where $\go\sim F$. Then we have
 \begin{align*}
 	\E[|t_{n}(G)-\E[t_{n}(G)]|^p]&\le  (2p)^{p/2} \E[L_n(G)^{p/2}]\E[\go^2]^{p/2}\notag\\
	&\qquad  + 2^{p/2} (2p)^{p-2}\E[L_n(G)] \E[\go^{p}]
 \end{align*}                                                     
where $L_{n}(G)$ is the number of pivotal edges for $t_{n}(G)$.
\end{lemma} 
\begin{proof}
	The proof essentially is a general version of the Efron-Stein inequality. Fix $n,G$ and a fixed enumeration $\{e_1,\ldots,e_N\}$ of the edges in $[n]\times G$ where $N$ is the number of edges in that graph. Consider the random variable $t_{n}(G)-\E[t_{n}(G)]$ as a function $f(\mvgo)$ of the edge weight configuration $\mvgo=(\go_1,\ldots,\go_N)\in \dR_+^N$ where $\go_i$ is the weight of the edge $e_i$. 
	
	 Let $\go_1',\ldots,\go_N'$ be i.i.d.~copies of $\go_1$. For a subset $S$ of $\{1,2,\ldots,N\}$ define $\mvgo^{S}\in\dR_+^N$ as the configuration where  $(\mvgo^{S})_i=\go_i$ for $i\notin S$ and $(\mvgo^{S})_i=\go'_i$ for $i\in S$. Recall that $[i]$ denote the set $\{1,2,\ldots,i\}$. Clearly $\mvgo^{[0]}=\mvgo$.  
	
	For illustration we will prove the $p=2$ case first  which is the Efron-Stein inequality. Recall that $\E[f(\mvgo)]=0$. We have
\begin{align*}
	\E[f(\mvgo)^2]	
	&= \E[f(\mvgo)(f(\mvgo) - f(\mvgo^{[N]}))]
	= \sum_{i=1}^N \E[f(\mvgo)(f(\mvgo^{[i-1]}) - f(\mvgo^{[i]}))].
\end{align*}     
Exchanging $\go_i,\go_i'$ one can easily see that $(\mvgo^{\{i\}},\mvgo^{[i]},\mvgo^{[i-1]})\equald (\mvgo,\mvgo^{[i-1]},\mvgo^{[i]})$ and hence we have
\begin{align*}
	\E[f(\mvgo)^2]	= \frac{1}{2} \sum_{i=1}^N \E[(f(\mvgo) - f(\mvgo^{\{i\}}) )(f(\mvgo^{[i-1]}) - f(\mvgo^{[i]}))].
\end{align*}
 By Cauchy-Schwarz inequality and exchangeability of $\go_i,\go'_i$ we see that 
\begin{align*}
	\E[f(\mvgo)^2] 
	\le \sum_{i=1}^N \E[(f(\mvgo) - f(\mvgo^{\{i\}}))^2\ind\{\go'_i>\go_i\}].
\end{align*}                                                       
 Now note that $\go'_i>\go_i$ and $f(\mvgo) \neq f(\mvgo^{\{i\}})$ implies that the $i$-th edge $e_i$ is essential for the configuration $\mvgo$ and moreover, $0<f(\mvgo^{\{i\}})- f(\mvgo)\le \go_i'-\go_i\le \go_i'$. Also $\go_i'$ is independent of $\mvgo$. Thus we have
\begin{align*}
	  \E[f(\mvgo)^2]	&\le \sum_{i=1}^N \E[(\go'_i)^2\ind\{e_i \text{ is essential for }\mvgo\}]
	=   \E[\go_i^2] \E[L_{n}]
\end{align*}                      
where $L_{n}$ is the number of pivotal edges for the configuration $\mvgo$.

Let $g(\cdot)$ be the function $g(x)=x|x|^{p-2}$. Using  similar decomposition as was done for $p=2$ case we have                      
\begin{align*}
	 \E[|f(\mvgo)|^p]  &=  \frac{1}{2}\sum_{i=1}^N \E[ (f(\mvgo) - f(\mvgo^{\{i\}}) )(g(\mvgo^{[i-1]}) - g(\mvgo^{[i]}) )].       
\end{align*}
Now   Lemma \ref{lem:tech1} and symmetry of $\go_i$ and $\go_i'$ imply that
\begin{align*}
	  \E[|f(\mvgo)|^p]  
	&\le a_{p}\sum_{i=1}^N \E\left[ | f(\mvgo) - f(\mvgo^{\{i\}}) |  | f(\mvgo^{[i-1]}) - f(\mvgo^{[i]}) |\right.\\
	&\qquad\qquad\left. \cdot \left(|f(\mvgo^{[i-1]})|^{p-2} + |f(\mvgo^{[i]})|^{p-2}\right)  \ind\{\go_i'>\go_i\} \right]
\end{align*}  
where $a_{p}=\max\{1,(p-1)/2\}$.
Note that $\go_i'>\go_i$, $f(\mvgo^{\{i\}})\neq f(\mvgo)$ and $f(\mvgo^{[i]}) \neq f(\mvgo^{[i-1]})$ imply that $0<f(\mvgo^{\{i\}})- f(\mvgo), f(\mvgo^{[i]}) - f(\mvgo^{[i-1]})\le \go_i' $ and the edge $e_i$ is essential for both the configurations $\mvgo$ and $\mvgo^{[i-1]}$. Moreover in that case we have 
\begin{align*}
	 |f(\mvgo^{[i]})|^{p-2}&\le |\ |f(\mvgo^{[i-1]})|+\go'_i|^{p-2}\\
	& \le 3 |f(\mvgo^{[i-1]})|^{p-2} + \max\{2,(2(p-3))^{p-3}\}(\go'_i)^{p-2}.  
\end{align*} 
The last line follows easily when $p\le 3$. For $p>3$ the last line follows by taking $\eps=e^{-1/(p-3)}$, using Jenson's inequality $(a+b)^{p-2}\le \eps^{3-p}x^{p-2}+(1-\eps)^{3-p}y^{p-2}$ and $(1-\eps)^{-1}\le \max\{2,2(p-3)\}$.  
Thus 
\begin{align*}
	  \E[|f(\mvgo)|^p]  
	&\le \sum_{i=1}^N \E\left[ (\go'_i)^2 \ind\{e_i \text{ is essential for } \mvgo^{[i-1]} \}\right.\\
	&\qquad\qquad \cdot \left. \left( 4 a_{p}|f(\mvgo^{[i-1]})|^{p-2} + b_{p}(\go'_i)^{p-2}\right) \right]
\end{align*} 
where  $b_{p}=a_{p}\max\{2,(2(p-3))^{p-3}\}$. Simplifying  we have 
\begin{align*}
	  \E[|f(\mvgo)|^p]  
	&\le \sum_{i=1}^N \E\left[ (\go'_i)^2 \ind\{e_i \text{ is essential for } \mvgo \} \left( 4a_{p}|f(\mvgo)|^{p-2} +b_{p} (\go'_i)^{p-2}\right) \right]\\
	&= 4a_p\E[(\go'_i)^2] \E[L_n |f(\mvgo)|^{p-2}] + b_p\E[(\go'_i)^{p}] \E[L_n]  
\end{align*}                                                                            
where $L_n$ is the number of pivotal edges in the configuration $\mvgo$. Let $y=\E[|f(\mvgo)|^p]^{(p-2)/p}$. Using H\"older's  inequality we have
\begin{align*}
	  y^{p/(p-2)}&= \E[|f(\mvgo)|^p]  \\
	&\le 4a_p \E[\go^2] \E[L_n^{p/2}]^{2/p} \E[|f(\mvgo)|^{p}]^{(p-2)/p} + b_p\E[\go^{p}] \E[L_n] \\
	&=  4a_p \E[L_n^{p/2}]^{2/p}  \E[\go^2]  y + b_{p} \E[L_n] \E[\go^{p}] .
\end{align*}                                                                 
Now Lemma \ref{lem:tech2} with $\gb=p/(p-2)$ gives that
\begin{align*}
	  \E[|f(\mvgo)|^p]^{2/p} &= y^{\gb-1} \le  
	  4a_p \E[L_n^{p/2}]^{2/p}\E[\go^2]  + ( b_p \E[L_n] \E[\go^{p}] )^{2/p}
\end{align*}
or
\[
   \E[|f(\mvgo)|^p] \le 2^{p/2-1} (2a_p)^{p/2} \E[L_n^{p/2}]\E[\go^2]^{p/2}  + 2^{p/2-1} b_p \E[L_n] \E[\go^{p}].  
\] 
Note that $2a_p\le p$ and $b_p\le 2^{p-1}p^{p-2}$. Hence simplifying we finally conclude that
\begin{align*}
	  \E[|f(\mvgo)|^p] \le  
	  (2p)^{p/2} \E[L_n^{p/2}]\E[\go^2]^{p/2}  + 2^{p/2} (2p)^{p-2} \E[L_n] \E[\go^{p}].
\end{align*} 
Now we are done.                                                                                                              
\end{proof}

 It is easy to see that $L_{n}(G)$ is  smaller than the length of any length minimizing path. In fact the random variable $L_{n}(G)$ grows linearly with $n$. The following well-known result due to Kesten~\cite{kes86} will be useful to get an upper bound on the length of a weight minimizing path. 

\begin{lemma}[Proposition $5.8$ in Kesten~\cite{kes86}]\label{lem:plen}
	If $F(0)<p_c(d)$  then there exist constants $0<a, b, c <\infty$ depending on $d$ and $F$ only, such that the probability that there exists a selfavoiding path $\cP$ from the origin which contains at least $n$ many edges but has $\go(\cP)<cn$ is smaller than $ae^{-bn}$. 
\end{lemma}  

Combining Lemma \ref{lem:mubd} and Lemma~\ref{lem:plen} we have the proof of Proposition~\ref{prop:mbd}.

\begin{proof}[\bf Proof of Proposition~\ref{prop:mbd}]
Note that $G_n=G$ for all $n$ clearly implies that $L_{n}(G)\le 3nk$ where $k=k(G)$ is the number of edges in $G$. This completes the proof for the case where the constants depend on $G$. 

Let $ \pi_n$ be the minimum number of edges in a weight minimizing path for $t_{n}(G_n)$. To complete the proof it is enough to show the following: if $G_n$'s are subgraphs of $\dZ^{d-1}$ and $F(0)<p_c(d)$ we have $\E[\pi_n^{p/2}]\le cn^{p/2}$ for some constant $c$ depending only on $d,p$ and $F$. We follow the idea from~\cite{kes93}. We have
\begin{align*}
	\pr(\pi_n>tn)&\le \pr(t_{n}(G_n)>ctn) + \pr(\text{there exists a self avoiding path $\cP$ }\\
	&\qquad \text{starting from $0$ of at least $tn$ edges but with $\go(\cP)<ctn$}). 
\end{align*}                                                      
Now using Lemma~\ref{lem:plen} we see that the second probability decays like $ae^{-btn}$. And the first probability is bounded by $\pr(S_n>ctn)$ where $S_n$ is the weight of the straight line path joining $(0,o)$ to $(n,o)$. Clearly $S_n$ is sum of $n$ many i.i.d.~random variables. Thus we have   
\begin{align*}
	\E[\pi_n^{p/2}] &=\int_0^\infty \frac{n^{p/2}p}{2}t^{p/2-1}\pr(\pi_n>tn)\; dt\\
	&\le \int_0^\infty \frac{n^{p/2}p}{2}t^{p/2-1}\pr(S_n>ctn)\; dt + \int_0^\infty \frac{n^{p/2}p}{2}t^{p/2-1}ae^{-btn}\; dt   \\
	&= c^{-p/2}\E[S_n^{p/2}] + \frac{ap}{2b^{p/2}}\Gamma(p/2)
	\le c_1n^{p/2}
\end{align*}      
where the constant $c_1$ depends on $d,p$ and $F$. The result for $a_{n}(G)$ and $T_{n}(G)$ follow by Lemma~\ref{lem:p2s} that 
\begin{align*}
	\E[|t_{n}(G)-a_{n}(G)|^p] &\le \E[|t_{n}(G)-T_{n}(G)|^p]\le \E[(2D\go)^p]
\end{align*} 
for all $n,G$ when $\E[\go^p]<\infty$ for some $p\ge 2$ with $ \go\sim F$ and $D$ is the diameter of $G$. 
\end{proof} 

\begin{proof}[Proof of the first technical Lemma \ref{lem:tech1}]
	For $x,y\in \Real/\{0\}, x\neq y$, let $z=x/y$. Then we have
	\[ \frac{x\abs{x}^{p-2} - y\abs{y}^{p-2} }{(x-y)(\abs{x}^{p-2}+\abs{y}^{p-2})} =\frac{z\abs{z}^{p-2}-1}{(z-1)(\abs{z}^{p-2}+1)}. \]
	Now, the lemma follows from the fact that
	\begin{align*}
	c_{p} :=\sup_{z\in \Real} \abs{\frac{z\abs{z}^{p-2}-1}{(z-1)(\abs{z}^{p-2}+1)}}
	\le\max\{1,(p-1)/2\}.
	\end{align*}
	To prove this note that, by $p>2$ we have \[  \sup_{z\ge0}   \frac{z^{p-1}+1}{(z+1)(z^{p-2}+1)} \le 1\]
	and
\begin{align*} 
	 \sup_{z\ge0}   \frac{{z^{p-1}-1}}{(z-1)(z^{p-2}+1)} 
	 &= \left( 1 - \sup_{x\ge 0}\frac{\sinh \frac{p-3}{p-1}x }{\sinh x}\right)^{-1}
	 =
\begin{cases}
	\left(1 - \frac{p-3}{p-1}\right)^{-1} &\text{ if } p>3,\\
	\left(1 - 0\right)^{-1} &\text{ if }  p  \le 3
\end{cases}
\end{align*}
and the line can be written succinctly as $\max\{1,(p-1)/2\}$.
\end{proof}                                 
 
\begin{proof}[Proof of the second technical Lemma \ref{lem:tech2}]
	Define $f(a,b):= (b+a^{1-1/\gb})^{1/(\gb-1)}$ and $g(a,b):=\sup\{y\ge 0: y^{\gb}\le a + by\}$. Without loss of generality assume $b>0$. Then it is easy to see that 
	\[
	g(a,b)= b^{1/(\gb-1)}g(ab^{-\gb/(\gb-1)},1)
	\text{ and } 
	f(a,b)=b^{1/(\gb-1)}f(ab^{-\gb/(\gb-1)},1).
	\]
	 So again w.l.g. we can assume that $b=1$.   Clearly $f(a,1)\ge 1, g(a,1)\ge 1$.

	Let $F:[1,\infty)\to \dR$ be the strictly increasing function $F(x):=x^{\gb}-x$. Note that $F(g(a,1))=a$. Now $y> f(a,1)$ implies that $y^{\gb}-y=F(y)> F(f(a,1))=f(a,1)(f(a,1)^{\gb-1}-1)\ge a^{1/\gb}(1+a^{(\gb-1)/\gb}-1)=a$. Hence the upper bound is proved.
\end{proof}


\section{Proof of Theorem~\ref{thm:maingr}} 
\label{sec:maingr} 

The proof of Theorem~\ref{thm:maingr} will be given in several steps. First we will show that it is enough to prove the CLT for $T_n(G_n)$ after proper centering and scaling. Then we will prove that $T_n(G_n)$ is ``approximately'' a sum of i.i.d.~random variables each having distribution $T_l(G_n)$ and an error term where $l$ depends on $n$. Finally, \change{writing $T_l(G_n)$'s inductively as approximate i.i.d.~sums} (the `renormalization steps') and controlling the error in each step, we will complete the proof. Recall that the notations $a_{n}=O(b_{n})$ and $a_{n}=o(b_{n})$, respectively, mean that $a_{n}\le C b_{n}$ for all $n\ge 1$ for some constant $C<\infty$ and $a_{n}/b_{n}\to 0$ as $n\to\infty$. Throughout the proof $c$ will denote a constant that depends only on $q,F$ and whose value may change from line to line.
   
\subsection{Reduction to $T_{n}(G_n)$} 
\label{sub:TnGn}
Let us first recall the setting. We have a sequence of nondecreasing graphs $G_n$ with $G_n$ having  diameter $d_n$ and $k_n$ edges. We also have $k_n=O(d_n^\theta)$ for some fixed $\theta\ge 1$. 
Define 
\[
\mu_n(G):=\E[T_{n}(G)] \text{ and } \gs_n^2(G):=\var(T_{n}(G))
\]                                                          
for any integer $n\ge 1$ and any finite connected graph $G$. 
 
Now from Lemma~\ref{lem:p2s} we have 
\[
	\E[|a_{n}(G_n)-T_{n}(G_n)|^p]\le 2^p d_n^p \E[\go^p]
\]                                                          
for all $n$ when $\E[\go^p]<\infty$ for a typical edge weight $\go$. Moreover, from Proposition~\ref{prop:varbd} we have $\gs_n^2(G_n)\ge cnk_n^{-1}$ for all $n$ for some absolute constant $c>0$ when $d_n=o(n)$. Thus when $d_n^2=o(nk_n^{-1})$ (which is satisfied if $d_n=o(n^{1/(2+\theta)})$), we have
\[
\frac{T_{n}(G_n)-\mu_n(G_n)}{\gs_n(G_n)} - \frac{a_n(G_n)-\E[a_n(G_n)]}{\var(a_n(G_n))^{1/2}} \stackrel{L^2}{\longrightarrow} 0.
\]
Hence it is enough to prove CLT for $(T_{n}(G_n)-\mu_n(G_n))/\gs_n(G_n)$ when $d_n=o(n^{1/(2+\theta)})$. From now on we will assume that 
\[
d_n=o(n^\ga)\text{ with }\ga< 1/(2+\theta) \text{ fixed}.
\]
       
\subsection{Approximation as an i.i.d.~sum} 
\label{sub:iidsum}
In Lemma \ref{lem:ulbd} we will prove a relation between side-to-side first-passage times in large and small cylinders and this will be crucial to the whole analysis. Fix an integer $n$ and a finite connected graph $G$. Let $n=ml+r$ with $0\le r<l$ where $l\ge 1$ is an integer. 

We  divide the cylinder graph $[n]\times G$ horizontally into  $m$ equal-sized smaller cylinder graphs $R_1,\ldots,R_m$ with  $R_i=[(i-1)l,il]\times G, i=1,2,\ldots,m$  each having width $l$ and a residual graph $R_{m+1}=[ml,n]\times G$. Let 
\begin{align}\label{def1}
	    X_i=T_{(i-1)l,il}(G)
\end{align}
be the side-to-side first-passage time for the product graph $R_i$ for $i=1,2,\ldots,m$ (see Definition~\ref{def:TabG}). We also define $X_{m+1}=T_{ml,n}(G)$ for the residual graph $R_{m+1}$. Clearly $X_{m+1}=0$ if $r=0$.  Note that $X_i$'s depend on $n$ and $G$, but we will suppress $n,G$ for readability.   We have the following relation. This is a generalization of Lemma~\ref{lem:p2s}.

\begin{figure}[htbf]
\includegraphics[width=4.5in]{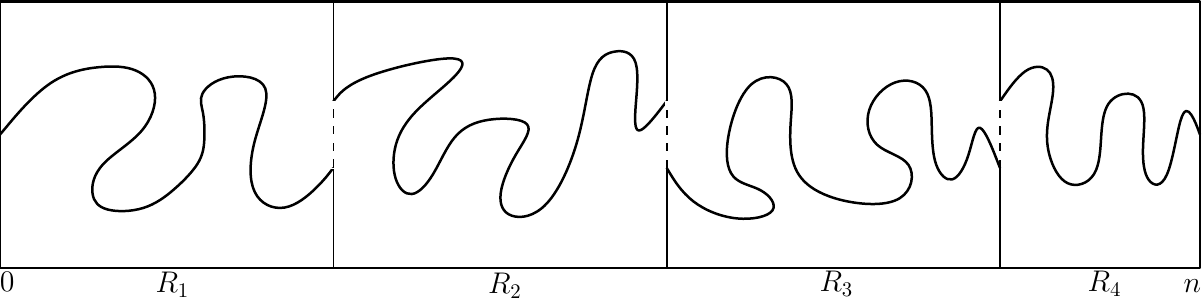}
\end{figure}

\begin{lemma}\label{lem:ulbd}
   Let $n,G$ be fixed. Let $X_i$ be as defined in \eqref{def1}. Then the random variable 
\[
    Y:= T_{n}(G) - ( X_1+X_2+\cdots+X_{m+1})
    \]
     is nonnegative and is stochastically dominated by $S_{mD}$ where $S_{mD}$ is sum of $mD$ many i.i.d.~random variables each having distribution $F$ and $D$ is the diameter of $G$.
	Moreover, $X_1,\ldots,X_m$ are i.i.d.~ having the same distribution as $T_{l}(G)$, $X_{m+1}$ has the distribution of $T_{r}(G)$ and $X_{m+1}$ is independent of $X_1,\ldots,X_m$.
\end{lemma}  
\begin{proof}
 First of all, it is easy to see that $X_i$ depends only on the  weights for the edge set $\{e: e \text{ is an edge in } [(i-1)l,il]\times G\}\setminus\{e\mid e \text{ is an edge in } \{(i-1)l\}\times G \text{ or } \{il\}\times G \}$. Thus, $X_1,\ldots,X_m$'s are i.i.d.~ having the same distribution as $T_{l}(G)$.        

Now choose a minimal weight path $\cP^*$ joining the left boundary $\{0\}\times G$ to the right boundary $\{n\}\times G$ (if there are more than one path one can use some deterministic rule to break the tie). The path $\cP^*$ hits all the boundaries $\{il\}\times G$ at some vertex for $i=0,1,\ldots,m$. Let $u_i,v_i, i=0,1,\ldots,m$ be the vertices in $G$ such that for each $i$, $\cP^*$ hits $\{il\}\times G$ for the last time at the vertex $(il,u_i)$ and 	after that it hits the boundary $\{(i+1)l\}\times G$ at the vertex $((i+1)l,v_i)$ for the first time (take $(m+1)l$ to be $n$). Clearly if $\cP^*$ hits $\{il\}\times G$ only at a single vertex then $u_i=v_{i-1}$. Now the part of $\cP^*$ between the vertices $(il,u_i)$ and $((i+1)l,v_i)$ is a path in $[il,(i+1)l]\times G$ and hence has weight more than $X_i$. But all these parts are disjoint. Hence we have $T_{n}(G)=\go(\cP^*)\ge \sum_{i=1}^{m+1} X_i$.  

Now to prove upper bound for $Y$, let $\cP^*_i$ be a minimal weight path joining the left boundary $\{il\}\times G$ to the right boundary $\{(i+1)l\}\times G$ and  achieving the weight $X_i$. Suppose $\cP^*_i$ hits $\{il\}\times G$ at $(il,w_i)$ and hits $\{(i+1)l\}\times G$ at $((i+1)l,z_i)$ for $i=0,1,\ldots,m$. Let $\cP_i$ be a minimal length path in $\{il\}\times G$ joining $(il,z_{i-1})$ to $(il,w_i)$ for $i=1,2,\ldots,m$. Consider the concatenated path $\cP_0^*,\cP_1,\cP_1^*,\cP_2,\ldots,\cP_{m}^*$ joining $(0,w_0)$ to $(n,z_{m+1})$. By minimality of weight we have 
\[
	T_{n}(G)\le \sum_{i=1}^{m}\left(X_i+\go(\cP_i)\right) +X_{m+1}.
\]                                                          
Thus we have $Y=T_{n}(G)- \sum_{i=1}^{m+1} X_i \le \sum_{i=1}^{m}\go(\cP_i)$. Clearly $\sum_{i=1}^{m}\go(\cP_i)$ is a sum of $\sum_{i=1}^{m} d(z_{i-1},w_i)$ many i.i.d.~random variables each having distribution $F$ where $d(\cdot,\cdot)$ is the graph distance in $G_n$. But we have $\sum_{i=1}^{m} d(z_{i-1},w_i)$ $\le$  $mD$ by definition of the diameter. Now $F$ is supported on $\dR^+$. Thus we are done.  
\end{proof}       
 
An obvious corollary of Lemma \ref{lem:ulbd} is the following.      
\begin{corollary}\label{cor:mvest} 
	For any integer $m,l,r$ and connected graph $G$ we have                           
\[
\abs{\mu_{ml+r}(G) - (m\mu_l(G)+\mu_r(G))} \le  {m D}\mu
\]
 and
\begin{align*}
	\abs{ \gs_{ml+r}(G)- (m\gs^2_l(G)+\gs_r^2(G))^{1/2} }\le mD(\mu^2+\gs^2)^{1/2}  
\end{align*}         
where $D$ is the diameter of $G$.
\end{corollary}   
\begin{proof}
	Taking expectation of $Y$ in Lemma~\ref{lem:ulbd} with $n=ml+r$ we have $\E[Y]=\mu_n(G) - m\mu_l(G) - \mu_r(G)$ and 
	$0\le \E[Y]\le mD\mu$.
	                                                           
	Moreover, we have 
	\begin{align*}
	   &\abs{\var(T_n(G))^{1/2} - \var(T_n(G)-Y)^{1/2}}\\ 
	&= \abs{\norm{T_n(G)-\E[T_n(G)]}_2 - \norm{T_n(G)-Y - \E[T_n(G)-Y]}_2}\\
	&\le \norm{Y-\E[Y]}_2\le (\E[Y^2])^{1/2}\le mD(\mu^2+\gs^2)^{1/2}.
	\end{align*}  
	Now the result follows since $ T_n(G)-Y =  \sum_{i=1}^{m+1} X_i$ and $X_i$'s are independent of each other.
\end{proof}


\subsection{Lyapounov condition} 
\label{sub:lyapounov_condition}
From here onwards, we return to using $n$ in subscripts and superscripts. From Lemma~\ref{lem:ulbd} and Corollary \ref{cor:mvest} clearly we have
\begin{align}
	   &\E|T_n(G_n)- \mu_n(G_n) - ( X_1^{(n)} +X_2^{(n)} +\cdots X_{m}^{(n)} - m\mu_l(G_n)) |\notag \\   
	&\le   \E|T_n(G_n) - ( X_1^{(n)} +X_2^{(n)} +\cdots X_{m+1}^{(n)})  |+md_n\mu+\E|X_{m+1}^{(n)}-\mu_r(G_n)| \notag\\
	&\le  2md_n\mu+\gs_r(G_n) \label{eq:ulm} 
\end{align}                 
where $X_i^{(n)},i=1,2,\ldots,m$ are defined as in \eqref{def1} 
and $n=ml+r$. We will take 
\[
	l=\max\{\lfloor n^\gb\rfloor,1\} \text{ for some fixed } \gb\in (2/(2+\theta),1) \text{ and } m=\lfloor n/l \rfloor.
\]       
Then we have $d_{n}^{2}=o(l)$ and all the lower and upper bounds on moments are valid for $T_{l}(G_{n})$.     
The dependence of $m,l$ on $n$ is kept implicit.  Note that $0\le r<l$. 
Moreover,  writing $l-r$ in place of $l$ and $1$ in place of $m$, we get from Corollary \ref{cor:mvest} that
\begin{align}
   \gs_r(G_n) 
	&\le \gs_l(G_n) + (\mu^2+\gs^2)^{1/2} d_n. \label{eq:errbd}
\end{align}        
Thus from \eqref{eq:ulm} we have 
\begin{align}
	   &\E\abs{ \frac{T_n(G_n)- \mu_n(G_n)}{\sqrt{m}\gs_l(G_n)} - \frac{\sum_{i=1}^m (X_i^{(n)} - \mu_l(G_n))}{\sqrt{m}\gs_l(G_n)} }\notag\\
	&\le \frac{2md_n \mu + \gs_r(G_n)}{\sqrt{m}\gs_l(G_n)} \le \frac{1}{\sqrt{m}} +  3(\gs^2+\mu^2)^{1/2}\frac{\sqrt{m} d_n}{\gs_l(G_n)}.   \label{eq:erb}
\end{align}
Recall that we have $l\sim n^\gb$ for some $\gb<1$ and thus $m\sim n^{1-\gb}$. From the lower bound for the variance in Proposition~\ref{prop:varbd} (as $d_{n}=o(l)$) we have 
\[
\frac{m d_n^2}{\gs^2_l(G_n)}\le \frac{cm^2d_n^2k_n}{n},
\]
where $c$ is some absolute constant. 
By our assumption on $m,d_{n}$ and $k_{n}$ we have $m^2d_n^2k_n=o(n)$ when
$
	      \ga \le (2\gb-1)/(2+\theta)
$
which is true for some $\gb<1$ as $\ga<1/(2+\theta)$.
  Hence  $(T_n(G_n)-\mu_n(G_n))/\sqrt{m}\gs_l(G_n)$ has the same asymptotic limit as 
\begin{align}
\frac{\sum_{i=1}^m X_i^{(n)} - m\mu_l(G_n)}{\sqrt{m}\gs_l(G_n)}\label{eq:2clt}
\end{align}  
 as $n\to\infty$ when
 \begin{align} \label{eq:kndn}
	      \ga \le \frac{2\gb-1}{2+\theta} \text{ for some } \gb\in \left(\frac{2}{2+\theta},1\right).
\end{align} 
 
Now $X_i^{(n)},i=1,2,\ldots,m$ are i.i.d.~random variables with finite second moment, hence by the CLT for triangular arrays  it is expected that 
\eqref{eq:2clt}                                                         
has standard Gaussian distribution asymptotically. However we cannot expect CLT for all values of $\gb$. 

 Let $s_n^2:=m\gs_l^2(G_n)$ be the variance of $\sum_{i=1}^m X_i^{(n)}$. To use Lindeberg condition for triangular arrays of i.i.d.~random variables we need  to show that
\begin{align*}
	\frac{m}{s_n^2}\E[\tilde{T}_l^2\ind\{|\tilde{T}_l|\ge \eps s_n\}]\to 0 \text{ as } n   \to\infty
\end{align*}  
for every $\eps>0$ where $\tilde{T}_l= T_{l}(G_n)-\mu_l(G_n)$. 
%
%
%
However, any bound using the relation $T_{l}(G_n)\le S_{l}$ where $S_l$ is the weight of the straight line path joining $(0,o)$ and $(l,o)$, gives rise to the condition $\theta \ga\le 1-2\gb$. The last condition is contradictory to \eqref{eq:kndn}. The difficulty arises from the fact that the lower and upper bounds for the variances are not tight.

Still we can prove a CLT by using estimates for the moments of $\tilde{T}_l(G_n)$ from Proposition~\ref{prop:mbd} and using a blocking technique which is reminiscent of the renormalization group method. Note that Lindeberg condition follows from the Lyapounov condition 
\begin{align}\label{eq:lya}
	\frac{m}{s_n^{p}}\E[|T_{l}(G_n)-\mu_l(G_n)|^p]\to 0 \text{ as } n\to \infty \text{ for some }p>2
\end{align}
and thus  it is enough to prove \eqref{eq:lya} for some $\gb\in (2/(2+\theta),1)$ where $l=\max\{\lfloor n^\gb\rfloor,1\}, m=\lfloor n/l\rfloor, s_n^2=m\gs_l^2(G_n)$. We also need to satisfy \eqref{eq:kndn} to complete the proof of Theorem~\ref{thm:maingr}.
                  

\subsection{A technical estimate} 
\label{sub:a_technical_estimate}
We need the following technical estimate for the next ``renormalization'' step.
The lemma gives an upper bound on the moment of sums of i.i.d.~random variables. It is known as Rosenthal's inequality (see~\cite{rose70}) in the literature.

\begin{lemma}\label{lem:iidbd}
	Let $Y_i,i=1,2,\ldots,m$ be i.i.d.~random variables with mean zero and  $\E[Y_i^{p}]<\infty$ for some  $p\ge 2$. Then we have
	\begin{align}
		\E[|Y_1+Y_2+\cdots+Y_m|^{p}]
	&\le A_{p}(  m\E[Y^{p}] + (m\E[Y^2])^{p/2}) 
	\end{align} 
	where $A_p$ is a constant depending only on $p$.                                    
\end{lemma}
\begin{proof} 
For simplicity we present the proof when $p=2q$ is an even integer.   
Let $Y\equald Y_1$ and $S_{m}=Y_{1}+\cdots+Y_{m}$. For $\mva=(a_1,a_2,\ldots,a_{2q})\in \dZ_+^{2q}$, we will denote $\sum_{i=1}^{2q}a_i$ by $|\mva|$ and $\sum_{i=1}^{2q}ia_i$ by $z(\mva)$. To estimate $\E[S_{m}^{2q}]$, we will use the following decomposition which is an easy exercise in combinatorics. We have
\begin{align*}
	\E[S_{m}^{2q}]&=\sum_{\mva \in \dZ_+^{2q}: z(\mva)=2q} \frac{(2q)!}{\prod_{i=1}^{2q}i!^{a_i}a_i!} (m)_{|\mva|}  \prod_{i=1}^{2q} \E[Y^i]^{a_i} 
\end{align*}        
where $(m)_k:=m!/(m-k)!\le m^{k}$.  
Note that here we used the fact that $Y_i$'s are i.i.d..                        
Since $\E[Y]=0$ we can and we will assume that $a_1=0$. Thus using H\"{o}lder's inequality we have
\begin{align*}
	\E[S_{m}^{2q}]
	&\le \sum_{z(\mva)=2q} \frac{(2q)!}{\prod_{i=2}^{2q}i!^{a_i}a_i!} (m)_{|\mva|}  \prod_{i=2}^{2q} \E[|Y|^i]^{a_i}    \\
	&\le  \sum_{z(\mva)=2q} \frac{(2q)!}{\prod_{i=2}^{2q}i!^{a_i}a_i!} m^{|\mva|} \prod_{i=2}^{2q} \E[Y^2]^{\frac{a_i(q-i/2)}{q-1}} \E[Y^{2q}]^{\frac{a_i(i/2-1)}{q-1}}    \\ 
	&\le  \sum_{z(\mva)=2q} \frac{(2q)!}{\prod_{i=2}^{2q}i!^{a_i}a_i!} (m^q  \E[Y^2]^q)^{\frac{|\mva|-1}{q-1}} (m\E[Y^{2q}])^{\frac{q-|\mva|}{q-1}}.
\end{align*} 
Note that $2|\mva|\le z(\mva)=2q$ as $a_1=0$.
Now using the fact that $x^{\ga}y^{1-\ga}\le \ga x + (1-\ga)y$ for all $x,y\ge0, \ga\in[0,1]$ we finally have
\begin{align} 
	\E[S_m^{2q}]
	&\le  A_{q}(m\E[Y^{2q}] + m^{q}\E[Y^2]^{q})    \label{eq:iidsum2}
\end{align} 
where
\begin{align*}
A_{q}:= \sum_{z(\mva)=2q} \frac{(2q)!}{\prod_{i=2}^{2q}i!^{a_i}a_i!}
\end{align*}
is a constant depending only on $q$.
\end{proof}

\subsection{Renormalization Step} 
\label{sub:renormalization_step}
Now we are ready to start our proof of the Lyapounov condition.
For simplicity we will write  $T_{l}(G)-\mu_{l}(G)$ imply as $\tilde{T}_{l}(G)$. 
Recall that
\[
\nu=\lim_{n\to\infty} \frac{\E[T_n(G_n)]}{n}.
\]

\begin{lemma}\label{lem:tclt}
	Suppose that $\nu>0$ and $\E[\go^p]<\infty$ for some $p>2$ where $\go$ is a typical edge weight. Suppose either $G_n=G$ for all $n$ or $G_n$'s are subgraphs of $\dZ^{d-1}$. Let $l=\max\{\lfloor n^\gb\rfloor,1\}$, $d_n=o(n^\ga)$ with $2\ga< \gb$ and $k_n=O(d_n^\theta)$ for fixed $\theta\ge 1$. Suppose that  there exist $t\ge 1$ real numbers $\gb_i, i=1,2\ldots,t$ such that  $2\ga<\gb_t<\gb_{t-1}<\cdots<\gb_1=\gb$ and we have
	\begin{align*}
		\ga &\le \frac{1 - 2(\gb_i-\gb_{i+1}) - (1-\gb_i)/q}{2+\theta} \text{ for all } i=1,2,\ldots,t-1,\\
		\text{ and } \ga &\le \frac{q-1}{q}\cdot \frac{1 -\gb_t}{\theta}
	\end{align*}                                             
	where   $q= p/2 $.	Then we have 
	\begin{align*}
		\frac{\sum_{i=1}^m X_i^{(n)} - m\mu_l(G_n)}{\sqrt{m}\gs_l(G_n)} \weakc N(0,1)
	\end{align*}                                                                 
	as $n\to \infty$ where $X_i^{(n)}$'s are i.i.d.~with $X_i^{(n)}\equald T_{l}(G_n)$.      
\end{lemma}  
\begin{proof}
	Since $X_i^{(n)},i=1,2,\ldots,m$ are i.i.d.~with mean $\mu_l(G_n)$  and variance $\gs^2_l(G_n)$ and $\E[\go^p]<\infty$ for some $p>2$, we can use the Lyapounov condition to prove the central limit theorem. We need to show that
\[
		  \frac{m}{s_n^{p}}\E[|\tilde{T}_{l}(G_n)|^p]\to 0 \text{ as } n\to \infty
\]
 where $s_n^2=m\gs_l^2(G_n)$. By the variance lower bound from Proposition~\ref{prop:varbd} we have 
\begin{align} \label{eq:snbd}
	s_n^2\ge c_1\frac{ml}{k_n}\ge c_2 \frac{n}{k_n}
\end{align} for some constants $c_i>0$ where $k_n$ is the number of edges in $G_n$. Using the moment bound from Proposition~\ref{prop:mbd} and lower bound on $s_n^2$ (note that $d_{n}^{2}=o(l)$) we have
 	\begin{align*}
		  \frac{m}{s_n^{p}}\E[|\tilde{T}_{l}(G_n)|^p] \le  \frac{c_p ml^{p/2}}{(n/k_n)^{p/2}} \le \frac{c_p ml^{p/2}k_n^{p/2}}{(ml)^{p/2}}=   \frac{c_p k_n^{p/2}}{m^{(p-2)/2}}.
	\end{align*}                                            
	Thus when $k_n=o(m^{1-2/p})$ or equivalently $\theta\ga\le (1-2/p)(1-\gb)$, we see that the right hand side converges to zero and we have a central limit theorem. This proves the assertion of the theorem when $t=1$.
	
	Let us now look into the bounds more carefully. The random variable $T_{l}(G_n)$ itself behaves like a sum of i.i.d.~random variables each having distribution $T_{l'}(G_n)$ for $l'<l$. We will use this fact to improve the required  growth rate of $k_{n}$. 
	Let $q=p/2$ and  assume that there exist $t\ge 2$ real numbers $\gb_i, i=1,2\ldots,t$ such that  $2\ga<\gb_t<\gb_{t-1}<\cdots<\gb_1=\gb$ and we have
	\begin{align}
	\begin{split}
		\ga &\le \frac{1 - 2(\gb_i-\gb_{i+1}) - (1-\gb_i)/q}{2+\theta} \text{ for all } i=1,2,\ldots,t-1\\
		\text{ and } \ga &\le \frac{q-1}{q}\cdot \frac{1 -\gb_t}{\theta} .
		\end{split}\label{eq:assump}
	\end{align}
	From now on we will write $l_1,m_1$ and $\gb_1$ instead of $l,m$ and $\gb$ respectively. Recall that we have $l_1=\max\{\lfloor n^{\gb_1}\rfloor,1\}$ and $d_n=o(n^\ga)$. 
	We will take 
	\[
	l_i=\max\{\lfloor n^{\gb_i}\rfloor,1\},  m_i=\lfloor l_{i-1}/l_i\rfloor \text{  for } i=2,\ldots,t.
	\]
The idea is as follows. First we will break the cylinder graph $[0,l_1]\times G_n$ into $m_2$ many equal sized graphs each of which looks like $[0,l_2]\times G_n$. Then we will break each of the new graphs again into $m_3$ many equal sized graphs each of which looks like $[0,l_3]\times G_n$ and so on. We will stop after $t$ steps. Our goal is to break the error term into smaller and smaller quantities and show that the original quantity is ``small'' when each of the final quantities are ``small''. Throughout the proof $q,t,\theta,\ga,\gb_i,i=1,2,\ldots,t$ are fixed. 

	For simplicity, first we will assume that 
	\[
	l_1=m_2m_3\cdots m_t l_{t}.
	\] 
	Under this assumption we have $m_il_i= l_{i-1}$ for all $i=2\ldots,t$.
Otherwise one has to look at the error terms which can be easily bounded using essentially the same idea and are considered in \eqref{eq:errterms}.\\

\noindent{\bf First Step.}	Let us start with the first splitting. We break the rectangular graph $[0,l_1]\times G_n$ into $m_2$ many equal sized graphs $[(i-1)l_2,i l_2]\times G_n$ for $i=1,2,\ldots,m_2$. Recall that we have $l_1=m_2 l_2$.  
	
	 Let $S_{m_2}=\sum_{i=1}^{m_2}X_i$ where $X_i=T_{(i-1)l_2,il_2}(G_n)-\mu_{l_2}(G_n)$. Recall that $X_i$'s are i.i.d.~having the same distribution as $\tilde{T}_{l_{2}}(G_n)$ where $\tilde{T}_{l}(G_n)={T}_{l}(G_n)-\mu_{l}(G_n)$. Let 
	$
	\eps_1=\eps_1(n) := {m_1}/{s_n^{2q}}.
	$
   We need to show the Lyapounov condition: 
\begin{align}
	\eps_1 \E[\tilde{T}_{l_1}(G_n)^{2q}]=o(1).  \label{eq:toshow1} 
\end{align}
From Lemma~\ref{lem:ulbd} we have 
	\[
		\E[|\tilde{T}_{l_1}(G_n) - S_{m_2}|^{2q}]\le c(m_2d_n)^{2q}\E[\go^{2q}]
	\]                                                                  
	for some constant $c>0$. Moreover,  Lemma~\ref{lem:iidbd} implies that
	\[
	 \E[S_{m_2}^{2q}]\le A_{q}(m_2^q\E[\tilde{T}_{l_2}(G_n)^{2}]^q+m_2\E[\tilde{T}_{l_2}(G_n)^{2q}]). 
	\]
	Thus we have
	\begin{align*}
		\eps_1&\E[\tilde{T}_{l_1}(G_n)^{2q}]\\
		&\quad\le c(\eps_1 (m_2d_n)^{2q} + \eps_1 m_2^q\E[\tilde{T}_{l_2}(G_n)^{2}]^q+\eps_1m_2\E[\tilde{T}_{l_2}(G_n)^{2q}]).
	\end{align*} 
Hence we need to show that
	\begin{align}
	   \eps_1 (m_2d_n)^{2q}&=o(1), \label{eq:ibd11}\\ 
	\eps_1 m_2^q \gs_{l_2}^{2q}(G_n) &=o(1) \label{eq:ibd12}\\ 
	\text{ and } \eps_1 m_2\E[\tilde{T}_{l_2}(G_n)^{2q}]&=o(1) \label{eq:ibd13} 
	\end{align}                                 
	
Using the variance lower bound~\eqref{eq:snbd} we have
	\begin{align*}
		\eps_1(m_2d_n)^{2q} &\le c\frac{m_1(m_2)^{2q} (d_n^2k_n)^q }{n^q}
		\le c\left(\frac{d_n^2k_n }{n^{1-2(\gb_1-\gb_2)-(1-\gb_1)/q}}\right)^q.
	\end{align*}                                                
Now \eqref{eq:ibd11} follows as $d_n^2k_n=o(n^{(2+\theta)\ga})$ and 
$ (2+\theta) \ga \le 1-2(\gb_1-\gb_2)-(1-\gb_1)/q$.
Moreover, Corollary~\ref{cor:mvest} with $l_1=m_2l_2$ implies that
\begin{align*}
	 (m_2 \gs_{l_2}^2(G_n))^{1/2} \le \gs_{l_1}(G_n) + c m_2 d_n. 
\end{align*}
Thus using the definition of $\eps_1=\eps_1(n)$ and the fact that $s_n^2=m_1 \gs_{l_1}^{2}(G_n)$ we have
\begin{align*}
	  \eps_1m_2^q \gs_{l_2}^{2q}(G_n) \le c( \eps_1 \gs_{l_1}^{2q}(G_n) + \eps_1(m_2 d_n)^{2q} )\le c\left( m_1^{1-q} +   \eps_1  (m_2d_n)^{2q} \right)
\end{align*}                    
and the right hand side is $o(1)$ as $q>1$ and by \eqref{eq:ibd11}. 
So the only thing that remains to be proved is that 
\[
	 \eps_1 m_2 \E[\tilde{T}_{l_2}(G_n)^{2q}]=o(1).
\]    

\noindent{\bf Induction step.}
From the above calculations in step $1$ the induction step is clear. Define
\[
	\eps_i=\eps_i(n)= \frac{m_1m_2\cdots m_i}{s_n^{2q}} \text{ for } i\ge 1.
\] 

\noindent{\bf Claim $1$.} We have
$
		\eps_i (m_{i+1}d_n)^{2q}=o(1)
$                          
for all $i< t$.\\
\noindent{\bf Proof of Claim $1$.}  Fix any $i$. Using definition of $\eps_i$ and the variance lower bound from \eqref{eq:snbd} we have 
\begin{align*}
	 \eps_i (m_{i+1}d_n)^{2q} = \frac{m_1\cdots m_i(m_{i+1}d_n)^{2q}}{s_n^{2q}}
	&\le c \frac{n^{1-\gb_i}m_{i+1}^{2q} (d_n^2k_n)^q}{n^q}     \\
	&= o\left( \left[\frac{n^{(2+\theta)\ga}}{n^{ 1-2(\gb_i-\gb_{i+1})-(1-\gb_i)/q }} \right]^q\right).
\end{align*} 
Now the claim follows by our assumption \eqref{eq:assump} that
$
   (2+\theta)\ga\le 1-2(\gb_i-\gb_{i+1})-(1-\gb_i)/q \text{ for all }i<t.
$

Our next claim is  the following.\\

\noindent{\bf Claim $2$.} We have 
$\eps_{i} m_{i+1}^q \gs_{l_{i+1}}^{2q}(G_n) =o(1)$                                                                
for all $i\ge 1$.\\
\noindent{\bf Proof of Claim $2$.}  
We will prove the claim by induction on $i$. We have already proved the claim for $i=1$ in \eqref{eq:ibd12}. Now suppose that the claim is true for some $i\ge 1$. Using Corollary~\ref{cor:mvest} for $l_{i+1}=l_{i+2}m_{i+2}$ we see that
\begin{align*}
   \eps_{i+1} (m_{i+2} \gs_{l_{i+2}}^2(G_n))^q
&\le c( \eps_{i+1}  \gs_{l_{i+1}}^{2q}(G_n) + \eps_{i+1}(m_{i+2} d_n)^{2q} )\\
&=c( \eps_{i} m_{i+1}  \gs_{l_{i+1}}^{2q}(G_n) + \eps_{i+1}(m_{i+2} d_n)^{2q} ).
\end{align*}
Hence we have $\eps_{i+1} (m_{i+2} \gs_{l_{i+2}}^2(G_n))^q=o(1)$ by Claim $1$ and the induction hypothesis as $q>1$. This completes the  proof.\\

\noindent{\bf Claim $3$.} For any $i\ge 1$,
$\eps_i\E[\tilde{T}_{l_i}(G_n)^{2q}]=o(1) $  if $\eps_{i+1}\E[\tilde{T}_{l_{i+1}}(G_n)^{2q}]=o(1).$\\
\noindent{\bf Proof of Claim $3$.} 
Assume that 
$
	   \eps_{i+1} \E[\tilde{T}_{l_{i+1}}(G_n)^{2q}]=o(1).
$
We write $\tilde{T}_{l_i}(G_n)$ as a sum of $S_{m_{i+1}}$  and an error term of order $m_{i+1}d_n$ where $S_{m_{i+1}}$ is sum of $m_{i+1}$ many i.i.d.~random variables each having distribution $\tilde{T}_{l_{i+1}}(G_n)$.  Using Lemma~\ref{lem:iidbd}, as was done in the first step, one can easily see that $\eps_i\E[\tilde{T}_{l_i}(G_n)^{2q}]=o(1) $ when 
\begin{align}
	\eps_i (m_{i+1}d_n)^{2q}&=o(1), \label{eq:ibdj1} \\
	\eps_{i} m_{i+1}^q\gs_{l_{i+1}}^{2q}(G_n) &=o(1)  \label{eq:ibdj2}  \\
	\text{ and } 
	\eps_{i} m_{i+1} \E[\tilde{T}_{l_{i+1}}(G_n)^{2q}]&=o(1).\label{eq:ibdj3}
\end{align}
Now Condition \eqref{eq:ibdj1} holds by Claim $1$, 
Condition \eqref{eq:ibdj2} holds by Claim $2$ and Condition \eqref{eq:ibdj3} holds by the hypothesis  as $\eps_{i+1}=\eps_i m_{i+1}$. 

Hence if we stop at step $t$, we see that the central limit theorem holds when
$
	   \eps_t \E[\tilde{T}_{l_t}(G_n)^{2q}]=o(1).
$
By the upper bound for the $2q$-th moment from Proposition~\ref{prop:mbd} (as $d_n^2=o(l_t)$)  we see that 
$
	 \eps_t \E[\tilde{T}_{l_t}(G_n)^{2q}] \le \eps_t l_t^q
$
and by the lower bound for the variance from \eqref{eq:snbd} we have
\begin{align*}
	 \eps_t l_t^q
	 &\le \frac{cm_1m_2\cdots m_t l_t^q k_n^q}{n^q} 
	 =\frac{ck_n^q}{(m_1m_2\cdots m_t)^{q-1}}= o\left(  \frac{n^{q\theta \ga}}{n^{(q-1)(1-\gb_t)}}\right).
\end{align*} 
The last condition also holds by our assumption \eqref{eq:assump} that $q\theta\ga\le (q-1)(1-\gb_t)$. Thus we are done when $l_1=m_2m_3\cdots m_t l_{t}$. 

Now, in general we have $l_{i-1}=m_il_i+r_i$ for $i=2,\ldots,t$ where $0\le r_i<l_i$ for all $i$.  Using the same proof used in the case when all $r_{i}=0$, one can easily see from Claim $3$, that we need to prove the extra conditions that
\begin{align}\label{eq:errterms}
	\eps_{i} \E[\tilde{T}_{r_i}(G_n)^{2q}]=o(1) \text{ for all } i=2,3,\ldots,t.
\end{align}
Fix $i\in\{2,3,\ldots,t\}$. If $r_i\le l_t$ then we are done since $\eps_i\le \eps_t$ and by Proposition~\ref{prop:mbd} we have $\E[\tilde{T}_{r_i}(G_n)^{2q}]\le c(d_n^{2q} + l_t^q)\le c_1l_t^q$. The last inequality follows since $2\ga< \gb_t$. Now suppose that $l_{j+1}\le r_i<l_j$ for some $j\ge i$.  Since we have $\eps_i\le \eps_j$ for $j\ge i$ working with $r_i$ instead of $l_j$ and using the same inductive analysis used before we have the required result \eqref{eq:errterms}.
\end{proof}

\subsection{Choosing the sequence} 
\label{sub:choosing_the_sequence}

To complete the proof of Theorem~\ref{thm:maingr} we need to choose an appropriate sequence $(\gb_1,\ldots,\gb_t)$ in \eqref{eq:assump} which will be provided by Lemma~\ref{lem:sys}. Note that 
\[
	\frac{1 - 2(\gb_0-\gb_{1}) - (1-\gb_0)/q}{2+\theta} = \frac{2\gb_{1} - 1}{2+\theta}
\]   
for $\gb_0=1$ and we have noted earlier in \eqref{eq:kndn} that 
\[
\frac{a_n(G_n)- \E[a_n(G_n)]}{\var(a_n(G_n))^{1/2}}
\text{ has the same asymptotic limit as }
\frac{\sum_{i=1}^m X_i^{(n)} - m\mu_l(G_n)}{\sqrt{m}\gs_l(G_n)}
\]
when $d_n=o(n^\ga)$ and 
$
\ga\le (2\gb_1 - 1)/({2+\theta}).
$

\begin{lemma}\label{lem:sys}
	Let $\gb_1,\gb_2,\ldots,\gb_t$ be $t$ real numbers satisfying the system of linear equations
	\begin{align}
		\frac{1 - 2(\gb_i-\gb_{i+1}) - (1-\gb_i)/q}{2+\theta}  =  \frac{q-1}{q}\cdot \frac{1 -\gb_t}{\theta} \label{eq:syseq}
	\end{align}
	for all $i=0,1,2,\ldots,t-1$ where $\gb_0=1$. Then we have                      
	\begin{align}
		\gb_i:= 1 - \frac{q\theta (1-r^i)}{\theta + (q-1)(2+\theta)(1-r^t)}
	\end{align}                                                            
	for all $i=1,2,\ldots,t$ where $r=1-1/(2q)$.
\end{lemma} 
\begin{proof}
Define $x_i=1-\gb_i$ for $i=0,1,\ldots,t$. Clearly $x_0=0$. Also define the constants
\[
	c=\frac{q-1}{q}\cdot \frac{2+\theta}{\theta} \text{ and } r=1-\frac{1}{2q}.
\]
Then the system of equations \eqref{eq:syseq} can be written in terms of $x_i$'s as
\begin{align} 
	1 - 2x_{i+1} + 2r x_i &=  c x_t \text{ for all } i=0,1,\ldots,t-1\notag\\
	\text{or } x_{i+1} - rx_{i} &= (1-c x_t)/2 \text{ for all } i=0,1,\ldots,t-1.   \label{eq:rec}
\end{align}
Multiplying the $i$-th equation by $r^{-i-1}$ and summing over $i=0,1,\ldots,t-1$ we have
\begin{align*}
	r^{-t}x_t =	qr^{-t}(cx_t-1)(r^t-1)
	\text{ or } x_t &=  \frac{q(1-r^t)}{1+qc(1-r^t)}.
\end{align*}
Now solving \eqref{eq:rec} recursively starting from $i=t-1,t-2,\ldots,0$ we have
\[
	x_i = \frac{q(1-r^i)}{1+qc(1-r^t)} \text{ for all }i=1,2,\ldots,t.
\]                                                                    
Simplifying and reverting back to $\gb_i$ we finally get 
\[
	x_i = 1- \frac{q\theta(1-r^i)}{\theta+(q-1)(2+\theta)(1-r^t)}
\]
for all $i=1,2,\ldots,t$.
\end{proof}


\subsection{Completing the proof} 
\label{sub:completing_the_proof}
                                                  
Now we connect all the loose ends to complete the proof of Theorem~\ref{thm:maingr}.
	
	Recall that the number of edges satisfies $k_n=O(d_n^\theta)$ and moreover we have $d_n=o(n^\ga)$ for some $\ga<1$. We also have $l\sim n^{\gb_1},m\sim n^{1-\gb_1}$ for some $\gb_1\in (\ga,1)$. We have proved in \eqref{eq:kndn} that the CLT will follow if we can find some $\gb_1\in (\ga,1)$ such that  
$\ga \le  ({2\gb_1 -1})/({2+\theta})$ and  
\begin{align}
	\frac{\sum_{i=1}^m X_i - m\mu_l(G_n)}{\sqrt{m}\gs_l(G_n)} \weakc N(0,1) \label{eq:xclt}
\end{align}                                                               
as $n\to\infty$ where $X_i$'s are i.i.d.~having distribution $T_l(G_n)$. Note that $(2\gb -1)/(2+\theta)<\gb/2$ for all $\gb>0$. 

To prove \eqref{eq:xclt} we will use the condition in Lemma~\ref{lem:tclt}. Assume that $\E[\go^{p}]<\infty$ for some real number $p>2$. Let $q=p/2$. From Lemma~\ref{lem:tclt} we see that CLT will hold in \eqref{eq:xclt} if  
there exist $t\ge 1$ real numbers $\gb_i, i=1,2\ldots,t$ such that  $2\ga<\gb_t<\gb_{t-1}<\cdots<\gb_1<\gb_0=1$ and 
\begin{align} 
\ga &\le \frac{q-1}{q}\cdot \frac{1 -\gb_t}{\theta} \text{ and }
	\ga \le \frac{1 - 2(\gb_i-\gb_{i+1}) - (1-\gb_i)/q}{2+\theta} \label{eq:sys}
\end{align} 
 for all $ i=0,1,\ldots,t-1$.  For $i=0$ the equation reduces to $\ga\le (2\gb_1-1)/(2+\theta)$.                                                 

Now fix any integer $t\ge 1$. Define 
$
r=1-1/2q.
$ 
For $i=1,\ldots,t$, define
\begin{align}
	\gb_i:= 1 - \frac{q\theta (1-r^i)}{\theta + (q-1)(2+\theta)(1-r^t)}.
\end{align}                                                             
As usual we will assume that $\gb_0=1$. Clearly $\gb_t<\gb_{t-1}<\cdots<\gb_1<\gb_0$. The sequence $(\gb_{1},\ldots,\gb_{t})$ is the unique solution to the system of equations given by equality in the right hand side of \eqref{eq:sys} (see Lemma~\ref{lem:sys}). In fact we have
\[
	\frac{q-1}{q}\cdot \frac{1 -\gb_t}{\theta}=    \frac{(q-1) (1-r^t)}{\theta + (q-1)(2+\theta)(1-r^t)}
\]  
and
\begin{align*}
	&\frac{1 - 2(\gb_i-\gb_{i+1}) - (1-\gb_i)/q}{2+\theta} =\frac{ (q-1)(1-r^t)}{\theta + (q-1)(2+\theta)(1-r^t)}  
\end{align*}     
for any $i=0,1,\ldots,t-1$. Now note that
\begin{align*}
	\frac{ 2(q-1)(1-r^t)}{\theta + (q-1)(2+\theta)(1-r^t)}  < 1 - \frac{q\theta (1-r^t)}{\theta + (q-1)(2+\theta)(1-r^t)}=\gb_t
\end{align*}
as $\theta+(q-1)(2+\theta)(1-r^t)- (2(q-1) + q\theta)(1-r^t) =\theta r^t>0$. 
Thus combining all the previous results we have 
\begin{align*}
	\frac{a_n(G_n)- \E[a_n(G_n)]}{\sqrt{m}\gs_l(G_n)}\weakc N(0,1) \text{ as }n\to\infty
\end{align*}
when
\[
	\ga \le \frac{ (q-1)(1-r^t)}{\theta + (q-1)(2+\theta)(1-r^t)}
\] 
for some integer $t\ge 1$. Since $r=1-1/(2q)<1$, letting $t\to\infty$ we get the  CLT  when
\[
	\ga < \frac{q-1}{\theta + (q-1)(2+\theta)}= \frac{1}{2+\theta + 2\theta/(p-2)}.
\]       
Thus we are done. \hfill\qed


\section{The case of fixed graph $G$} 
\label{sec:fixedg}   

By the arguments given in Section~\ref{sec:main}, we have a Gaussian central limit theorem for $a_n(G)$ and $T_{n}(G)$ as $n\to\infty$ after proper scaling when $G$ is a fixed graph. 
Proposition~\ref{prop:meanbd} says that 
\[
\nu(G):=\lim_{n\to\infty} \frac{\E[T_{n}(G)]}{n}
\]
exists and is positive.
Moreover, Proposition~\ref{prop:varbd} gives that 
\[
	0<c_1\le \frac{\var(T_{n}(G))}{n}\le c_2
\]         
for all $n$ for some constants $c_1,c_2>0$ depending on $G$. The next lemma says that in fact we can say more. Assume that $v(G)$ is the number of vertices in $G$, $k(G)$ is the number of edges in $G$ and $D=D(G)$ is the diameter of $G$.
	
\begin{lemma}\label{lem:fixedG}
	      Let $G$ be a finite connected graph. Then we have
	\[
		|\E[T_{n}(G)] - n\nu(G)|\le \mu D \text{ for all } n 
	\]
	and the limit
	\[
		\gs^2(G):=\lim_{n\to\infty} \frac{\gs_n^2(G)}{n}
	\]                              
exists and is positive.
\end{lemma} 
\begin{proof}
	Let $\tilde{\mu}_n=\mu_n/n$ and $\tilde{\gs}^2_n=\gs_n^2/n$. Using the proof given in corollary \ref{cor:mvest} we have
	\begin{align}
		|n\tilde{\mu}_n - (ml\tilde{\mu}_l+r\tilde{\mu}_r)|\le m\mu D \text{ and }
		\abs{(n \tilde{\gs}^2_n)^{1/2}- (ml\tilde{\gs}^2_l+r\tilde{\gs}_r^2)^{1/2} }\le mb D\label{eq:vbd}
	\end{align}         
	for all $n=ml+r$ with $0\le r<l$ where $b=(\mu^2+\gs^2)^{1/2}$. 
	Thus for any $m,k$ we have
	$
		|\tilde{\mu}_{mk} - \tilde{\mu}_m|\le {\mu D}/{m}.
	$
	Reversing the roles of $m$ and $k$, and combining, we see that for any $m,k$, we have
	\[
	|\tilde{\mu}_{m} - \tilde{\mu}_k|\le {\mu D}/{k} + {\mu D}/{m}.
	\]
	Taking limits as $k\to\infty$ we have, for any $m$,
	\[
	|\tilde{\mu}_m - \lim_{n\to \infty} \tilde{\mu}_n|\le {\mu D}/{m}.
	\]

	For the variance, we take $n=2l$ in equation~\eqref{eq:vbd} to have 
	\[
		\abs{\tilde{\gs}_{2l}- \tilde{\gs}_l }\le bD(2/l)^{1/2}.
	\]                                                                   
	Hence, it follows that $\tilde{\gs}_{2^k}$ is Cauchy and $\lim_{k\to\infty}\tilde{\gs}_{2^k}$ exists. 
	
	Now take any $l\ge 1$. There exists a unique positive integer $k=k(l)$ such that $2l^{3/2}\le 2^k < 4l^{3/2}$ ($k(l)=1 + \lceil \log_2 l^{3/2}\rceil$).  Suppose $2^k=ml+r$ where $0\le r<l$. Clearly $\sqrt{l}\le m\le 4\sqrt{l}$. Now from \eqref{eq:vbd} we have, 
	\begin{align*}
		\abs{(2^k\tilde{\gs}_{2^k}^2)^{1/2} - (ml\tilde{\gs}_l^2 + r\tilde{\gs}_r^2)^{1/2}} \le mbD.
	\end{align*}  
	Dividing by $2^{k/2}$ on both sides, we get
	\[
	\abs{\tilde{\gs}_{2^k} - \biggl(\tilde{\gs}_l^2 + \frac{r(\tilde{\gs}_r^2-\tilde{\gs}_l^2)}{ml+r}\biggr)^{1/2}} \le \frac{mbD}{\sqrt{ml+r}}\le  2bD l^{-1/4}.
	\]
	Note that $k,m,r$ are functions of $l$ in the above expression. Among these, $m(l)\ge l^{1/2}$ and $r(l) < l$. Taking $l \to \infty$, and using the fact that the sequence $\{\tilde{\gs}_n^2\}_{n\ge 1}$ is uniformly bounded (see Proposition~\ref{prop:mbd}), we get that   $\lim_{m\to \infty} \tilde{\gs}_m$ exists and equals $\lim_{k\to \infty} \tilde{\gs}_{2^k}$. Positivity of the limit follows from the variance lower bound given in Proposition~\ref{prop:varbd}.
\end{proof}

Note that, if we consider the point-to-point cylinder first-passage time $t_n(G)$ in $[0,n]\times G$, the same results given in Lemma~\ref{lem:fixedG} hold for $\E[t_n(G)]$ and $\var(t_n(G))$. 

Now we consider the process $X(m)$ where $X(m)=t_m(G)-m\nu(G)$ for $m\in \{0,1,\ldots\}$ and $X_n(t)=X_m + (t-m)(X_{m+1} - X_m )$ for $t\in (m,m+1)$. Note that when $G$ is the trivial graph consisting of a single vertex, $X(n)$ corresponds to random walk with linear interpolation and by Donsker's theorem $\{(n\gs^2)^{-1/2}X(nt)\}_{t\ge 0}$ converges to Brownian motion. The next lemma says that for general $G$ we also have the same behavior. We assume that  $\E[\go^p]<\infty$ for some $p>2$ where $\go\sim F$.

\begin{lemma}\label{lem:bm}
	The scaled process $\{ (n\gs^2(G))^{-1/2} X(n t) \}_{ t\ge 0}$ converges in distribution to standard Brownian motion as $n\to\infty$. 
\end{lemma}  
\begin{proof}
Consider the continuous process $X'$ defined as $X'(n):=T_{n}(G)-n\nu(G)$ for $n\in\{0,1,\ldots\}$ and extended by linear interpolation. By Lemma~\ref{lem:p2s} it is enough to prove Brownian convergence for $\{Y_n(t):=(n\gs^2(G))^{-1/2} X'(n t) : {0\le t\le T}\}$ for any fixed $T>0$. 
	To prove the result it suffices to show that the finite dimensional distributions of $Y_n(t)$ converge weakly to those of $B_t$ and that $\{Y_n\}$ is tight.

First of all note that for any $s>0$, we have
\begin{align*}
	|Y_n(s)- (n\gs^2(G))^{-1/2} X(\lfloor ns \rfloor)|
	&\le (n\gs^2(G))^{-1/2} |X'(1+\lfloor ns \rfloor) - X'(\lfloor ns \rfloor) |\\
	&\le (n\gs^2(G))^{-1/2} (Z + \nu(G))\probc 0
\end{align*}                  
where $Z$ is the maximum of all the edge weights connecting $\{\lfloor ns \rfloor\}\times G$ to $\{1+\lfloor ns \rfloor\}\times G$, which has the distribution of maximum of $v(G)$ many i.i.d.~random variables each having distribution $F$. Thus it is enough to prove finite dimensional distributional convergence of the process $\{ W_n(t):=(n\gs^2(G))^{-1/2} X'(\lfloor nt \rfloor) \}_{t\ge 0}$. For a fixed $t>0$, using Theorem~\ref{thm:maingr} we have $W_n(t)\weakc N(0,t)$ since $\lfloor nt \rfloor/n\to t$. 

For $0=t_0<t_1<t_2<\cdots<t_l<\infty$, define $V_i=T_{\lfloor nt_{i-1} \rfloor,\lfloor nt_i \rfloor}(G) - (\lfloor nt_{i}\rfloor- \lfloor nt_{i-1} \rfloor)\nu(G)$ for $i=1,2,\ldots,l$. Clearly $V_i$'s are independent for all $i$. Moreover using Lemma~\ref{lem:ulbd} we have  
\begin{align*}
	\E[|W_n(t_i)-W_n(t_{i-1}) - (n\gs^2(G))^{-1/2}V_i|]\to 0
\end{align*}                                                
as $n\to\infty$ for all $i$. Thus by independence and by CLT for $(n\gs^2(G))^{-1/2}V_i$, we have
\[
   (W_n(t_i)-W_n(t_{i-1}))_{i=1}^l \weakc (B_{t_i}-B_{t_{i-1}})_{i=1}^l \text{ as } n\to\infty. 
\]

To prove tightness for $\{Y_{n}(\cdot)\}$, first of all note that certainly $\{Y_n(0)\}$ is tight as $Y_n(0)\equiv 0$. Also it is enough to prove tightness for $\{W_n(\cdot)\}$. We will prove tightness via the following lemma.
\begin{lemma}[Billingsley~\cite{bl68}, page 87-91]\label{lem:tight}
	The sequence $\{W_n\}$ is tight if there exist constants $C\ge 0$ and $\gl>1/2$ such that for all $0\le t_1<t_2<t_3$ and for all $n$, we have 
	\[
		\E[|W_n(t_2)-W_n(t_1)|^{2\gl}  |W_n(t_3)-W_n(t_2)|^{2\gl}]   
		\le C |t_2-t_1|^\gl |t_3-t_2|^\gl.
	\]
\end{lemma}
Using the Cauchy-Schwarz inequality and Proposition~\ref{prop:mbd}, it is easy to that Lemma~\ref{lem:tight} holds with $\gl=p/4$. Thus we are done.
\end{proof}

\section{CLT upto the height threshold}\label{sec:clt_ht}

In this section we prove the Central Limit Theorem all the way upto the height threshold under a few natural but unproved assumptions. A close look at the proof of Theorem \ref{thm:maingr} shows two main sources of error: (i) error coming from the gap between the upper and lower bound of the variance, and (ii) error coming from the sum of vertical edge weights needed to join the within block optimal paths. To control the errors optimally we assume the existence of fluctuation and transversal exponents. Recall that, given any direction $\vx\in\dZ^d$, $a(\mvzero,\vx)$ denotes the length of the geodesic joining $\mvzero$ and $\vx$, and $D(\mvzero,\vx)$ denotes the euclidean distance between the geodesic path and the straight line path joining $\mvzero$ and $\vx$. We assume the following:

\begin{ass}[Existence of fluctuation exponent]\label{ass:1}
There exists a number $\chi\ge 0$ such that for every $\chi'>\chi$ there exists $\ga>0$ such that 
\[
\sup_{\vx\in\dZ^d\setminus \{\mvzero\}} \E\exp\left( \ga \cdot\frac{|a(\mvzero,\vx)-\E a(\mvzero,\vx)|}{|\vx|^{\chi'}}\right)<\infty
\]
and for every $\chi''<\chi$ we have
\[
\inf_{\vx\in\dZ^d\setminus \{\mvzero\}} \frac{\var(a(\mvzero,\vx))}{|\vx|^{2\chi''}} >0.
\]
\end{ass}

\begin{ass}[Existence of transversal exponent]\label{ass:2}
There exists a number $\xi\ge 0$ such that for every $\xi'>\xi$ there exists $\ga>0$ such that 
\[
\sup_{\vx\in\dZ^d\setminus \{\mvzero\}} \E\exp\left( \ga \cdot\frac{D(\mvzero,\vx)}{|\vx|^{\xi'}}\right)<\infty
\]
and for every $\xi''<\xi$ we have
\[
\inf_{\vx\in\dZ^d\setminus \{\mvzero\}} \frac{\E( D(\mvzero,\vx))}{|\vx|^{\xi''}} >0.
\]
\end{ass}

Roughly Assumptions \ref{ass:1} and \ref{ass:2} state that $\var(a(\mvzero,\vx))\approx |\vx|^{2\chi}$ and $D(\mvzero,\vx)\approx |\vx|^{\xi}$ for  $|\vx|$ large enough. Though there exists no rigorous proof on the existence of the fluctuation and transversal exponents, it seems quite reasonable to expect that if the two exponents $\chi$ and $\xi$ indeed exist, then they should satisfy the above properties. Moreover, it is not difficult to prove (see Chatterjee~\cite{chatterjee11}) that if such exponents exist then $0\le \xi\le 1$ and $0\le \chi\le \frac{1}{2}$. We also recall that the limiting shape $B_{0}$ given by
\[
B_{0}:=\{\vx\in\dR^{d}: \nu(\vx)\le 1\}
\]
where $\nu(\vx)$, as given in \eqref{hw}, is the asymptotic speed in the direction of $\vx$. We assume that $B_{0}$ has a positive curvature in the direction of $\mve_{1}$. 

\begin{ass}[Positive curvature]\label{ass:3}
Let $H$ be the $(d-1)$-dimensional plane passing through the point $\mve_{1}$ perpendicular to the line joining $\mvzero$ and $\mve_{1}$. There exists a positive constant $C$ such that for all $\vz\in H$ we have
\[
|\nu(\mve_{1}+z)-\nu(\mve_{1})|\le C|\vz|^{2}.
\]
\end{ass}

\begin{rem}
We assume the positive curvature in the direction of $\mve_{1}$ as we are trying to prove the CLT in that direction. In general, one can prove the existence of a direction $x_{0}$ such that in the $x_{0}$ direction the positive curvature assumption holds. 
\end{rem}

Under assumption \ref{ass:1}, \ref{ass:2} and \ref{ass:3} one can prove the following result.

\begin{lem}[KPZ scaling relation~\cites{chatterjee11,am11}]
Assume \ref{ass:1}, \ref{ass:2} and \ref{ass:3}. Then we have $\chi=2\xi-1$. 
\end{lem}

\begin{rem}
If we apriori assume that the KPZ scaling relation is true, then in assumption \ref{ass:1} and \ref{ass:2}, instead of all $\vx\in\dZ^{d}\setminus\{\mvzero\}$ it is enough to consider $\vx$ in a cone in the direction of $\mve_{1}$ with angle $\eps$ for small enough $\eps>0$.
\end{rem}


Now to control the error coming from the vertical fluctuation, we use a different blocking method. Instead of dividing the length $n$ cylinder $[n]\times [-h_{n},h_{n}]^{d-1}$ into small cylinders of equal length, we use two types of cylinders.  Big cylinders are of length $\ell_{1}$ and small cylinders are of length $\ell_{2}$ where $\ell_{2}\ll \ell_{1}$. Let $m=\lfloor n/(\ell_{1}+\ell_{2})\rfloor\approx n/\ell_{1}$. Divide the rectangle $[n]\times [-h_{n},h_{n}]^{d-1}$ into $2m+1$ many sub-rectangles $R_{i}$'s, where $R_{1},R_{3},R_{5},\ldots$ are small cylinders and the rest are big cylinders. The last one is the residual cylinder, which, for simplicity we will assume, is small. 

\begin{figure}[htbf]
\centering
\includegraphics[width=140mm]{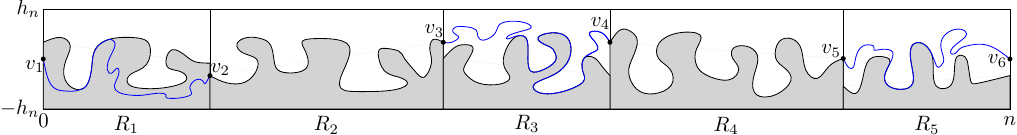}
\caption{Two type blocking }
\label{fig:2type}
\end{figure}

Let $X_{i}$ be the minimal passage time over all paths inside $R_{i}$ connecting the left and right boundaries of $R_{i}$  for $i=1,2,\ldots,2m+1$. For $i=1,2,\ldots,m$, let $v_{2i}$ and $v_{2i+1}$ be the left and right endpoint of the optimal path in $R_{2i}$ (see figure \ref{fig:2type}). Define $v_{1}=\mvzero$ and $v_{2m+2}=(n,0,\ldots,0)$. Let $Y_{i}$ be the first-passage time from $v_{i}$ to $v_{i+1}$ inside the cylinder $R_{i}$ for $i=1,2,\ldots,2m+1$. Clearly $Y_{2i}=X_{2i}$ for $i=1,2,\ldots,m$. 

Note that, $X_{1},X_{3},X_{5},\ldots$ are i.i.d.~and so is  $X_{2},X_{4},X_{6},\ldots$. Moreover they are independent of each other. On the other hand, $Y_{3},Y_{5},\ldots$ are identically distributed but not independent of each other. The main idea behind the above blocking technique, is to separate the height fluctuation and total passage-time fluctuation. While the error arising from height fluctuation will come from the small cylinders, the main contribution in the first-passage time fluctuation is coming from the big cylinders. 

As before, let $a_n(h_n)$ denote the first-passage time from $(0,0,\ldots,0)$ to $(n,0,\ldots,0)$ inside the rectangle  $[n]\times [-h_{n},h_{n}]^{d-1}$. Clearly we have 
\[
a_n(h_n)\ge X_{1}+X_{2}+X_{3}+X_{4}+\cdots+X_{2m}+X_{2m+1}.
\]
In the proof of Theorem \ref{thm:maingr} we used small rectangles of length $0$, \ie\ a vertical line. We also have,
\[
a_n(h_n)\le Y_{1}+X_{2}+Y_{3}+X_{4}+\cdots+X_{2m}+Y_{2m+1}.
\]

For the cylinder $C=[\ell]\times[-h_{n},h_{n}]^{d-1}$ we define the non-negative random variable
\[
\gD(C):=\max_{\vx\in B_{L},\vy\in B_{R}} T(\vx,\vy) -\min_{\vx\in B_{L},\vy\in B_{R}} T(\vx,\vy)
\]
where $B_{L}$ is the left boundary wall $\{\mvzero\}\times [-h_{n},h_{n}]^{d-1}$, $B_{R}$ is the right boundary wall $\{\ell\}\times [-h_{n},h_{n}]^{d-1}$ and $T(\vx,\vy)$ is the minimum passage time from $\vx$ to $\vy$ inside the cylinder $C$. Thus we have
\begin{align*}
0\le a_n(h_n)-(X_{1}+X_{2}+\cdots+X_{2m+1}) \le \sum_{i=1}^{m} (Y_{2i-1}-X_{2i-1}) \le \sum_{i=1}^{m} \gD(R_{2i-1}).
\end{align*}

In particular, as $\gD(R_{2i-1})$'s are i.i.d., we have 
\begin{align}\label{eq:kmbd}
\norm{(a_n(h_n)-\E[a_n(h_n)])- \sum_{i=1}^{2m+1}(X_{i}-\E[X_{i}])}_{k} \le 2m\cdot ||\gD(R_{1})||_{k}
\end{align}
for all $k\ge 1$ where $||X||_{k}=(\E|X^{k}|)^{1/k}$ for a random variable $X$. We prove the following lemma.

\begin{lem}\label{lem:deltabd}
Assume conditions \ref{ass:1}--\ref{ass:3}. Fix $\xi'>\xi$. Consider the cylinder $R=[n]\times [-h_{n},h_{n}]^{d-1}$ where $h_{n}=\Theta(n^{\xi'})$. Then there exists a constant $c>0$ such that  
\[
||\gD(R)||_{k} \le \frac{ckh_n^2}{n} \text{ for all $k\ge 1$.}
\]
\end{lem}

We also need a matching lower bound for the variance of $a_{n}(h_{n})$ to complete the program. Define $\gs_n^{2}(h_{n}):=\var(X_{1}+X_{2}+\cdots+X_{2m+1})=(m+1)\var(X_1)+m\var(X_2)$. From equation \eqref{eq:kmbd} it easily follows that
\begin{align}\label{eq:vsum}
|\sqrt{\var(a_{n}(h_{n}))}-\gs_{n}(h_{n})|\le 2m||\gD(R_{1})||_{2}.
\end{align}
Thus we can approximate $a_{n}(h_{n})$ by the sum $X_{1}+\cdots+X_{2m+1}$ when $m||\gD(R_{1})||_{2}\le cmh_{n}^{2}/n \ll \gs_{n}(h_{n})$. To get the appropriate lower bound for the variance we assume a natural technical condition that we are unable to prove. Let 
\begin{align}\label{xnh}
X(n,h) := &\text{ the minimum passage time from the left boundary to the }\notag\\ &\text{  right boundary inside the cylinder } [n]\times[-h,h]^{d-1}.
\end{align}

\begin{ass}\label{ass:4}
There exists $n_{0}>0$ such that for all fixed $n\ge n_{0}$ the function
\[
f(h)=\var(X(n,h))
\]
is a non-increasing function of $h$.
\end{ass}

Define 
\begin{align}\label{def:th}
\theta=\frac{1-2\chi}{\xi}\ge 0.
\end{align}
\change{Under the above four assumptions we have following moment bound. Note that the bound actually interpolates between the two cases: for $h=O(1)$ the fluctuation is of the order of $n^{1/2}$ and for $h\approx n^{\xi}$ the fluctuation is of the order of $n^{\chi}$.}
 
\begin{lem}\label{lem:clbd}
Assume conditions \ref{ass:1}, \ref{ass:2} and \ref{ass:3}. Let $X(n,h)$ be as in equation \eqref{xnh}. Then for any integer $k\ge 2$ and $\eps>0,\gd>0$, there exist constants $C>0, \xi'>\xi$ such that
\[
||X(n,h)-\E[X(n,h)]||_{k} \le C \sqrt{nh^{-\theta+\eps}}
\]
for all $n^{\gd}\ll h\ll n^{\xi'}$ where $\theta$ is as in \eqref{def:th}. 

Moreover, if we assume condition \ref{ass:4} and $\chi>0$, for every $\eps>0,\gd>0$ there exists a constant $c>0$ such that 
\[
\var(X(n,h)) \ge cnh^{-\theta-\eps}
\]
for all $h\gg n^{\gd}$. 
\end{lem}

\change{Now note that, when $\ell_{1}\approx n^{1-\gb}, \ell_{2}\approx h_{n}^{1/\xi'},\chi>0$, we have $m\approx n/\ell_{1}$, $||\gD(R_{1})||_{2}\le ch_{n}^{2}/\ell_{2}$, }
\[
m||\gD(R_{1})||_{2}\ll \frac{n}{\ell_1}\cdot h_n^{2-\frac{1}{\xi'}}\text{ and } \sqrt{nh_{n}^{-\theta-\eps}}\ll \gs_{n}(h_n).
\]
Writing $1/\xi'=1/\xi-\gd$, the sum approximation \eqref{eq:kmbd} is valid when
\begin{align*}
\frac{n}{\ell_1}\cdot h_n^{2-\frac{1}{\xi'}}\ll \sqrt{nh_{n}^{-\theta-\eps}}
\quad \text{ or }\quad  h_n^{2(2\xi-1)/\xi+\theta+\gd+\eps}\ll n^{1-2\gb}.
\end{align*}
Using the result that $\chi=2\xi-1$ we need
\[
h_n^{1/\xi+\gd+\eps}\ll n^{1-2\gb}
\]
which gives the condition $h_{n}\ll n^{\xi}$ as $\eps,\gd,\gb$ can be made arbitrarily small. When $\chi=0$, the variance lower bound is still valid but with $\theta$ replaced by $(d-1)$ (see Proposition~\ref{prop:varbd}) and we can proceed as before to get the condition $h_{n}\ll n^{1/(d-1)}$. Combining we have the following main result.

\begin{thm}\label{thm:mainht}
Assume conditions \ref{ass:1}, \ref{ass:2}, \ref{ass:3} and \ref{ass:4}. Let $\{h_n\}_{n\ge 1}$ be a sequence of integers satisfying $h_n=o(n^\ga)$ where  
\begin{align*}
	 \ga < \begin{cases}
	 \ \; \xi & \text{ if } \chi>0\\
	 \frac{1}{d-1} & \text{ if } \chi=0.
	 \end{cases}
\end{align*}     
Then we have
	\[
		\frac{a_n(h_n) - \E[a_n(h_n)]}{\sqrt{\var(a_n(h_n))}} \weakc N(0,1) \text{ as } n\to\infty.
	\]    
	Moreover, for any $\eps>0$ there exist constants $c,C>0$ such that
	\[
	 cnh_{n}^{-\theta-\eps}\le \var(a_{n}(h_{n}))\le Cnh_{n}^{-\theta+\eps}.
	\]     
\end{thm}

In dimension $2$ the conjectured values of the exponents are $\xi=1/3,\xi=2/3$. Thus the conjectured value of $\theta$ is $1/2$ which matches with the simulation results.  Moreover, for $d=2,3$, $1/(d-1)\ge 1/2$ and $\chi=0$ implies $\xi=1/2$. Thus we have the following corollary. 

\begin{cor}
Under the assumptions \ref{ass:1}, \ref{ass:2}, \ref{ass:3} and \ref{ass:4}, we have CLT for  $a_n(n^{\ga})$ in dimension $2$ and $3$ for $\ga<\xi$.  
\end{cor}

\section{Proof of CLT upto the height threshold}\label{sec:pclt}
Throughout the proof $C$ will denote a positive constant that depends only on the edge weight distribution and the dimension and may change from line to line. Let 
\begin{align}
\label{eta}
\eta(\vx):=\E[a(\mvzero,\vx)]
\end{align}
 for all $\vx\in\dZ^{d}$. Recall that $\nu(\vx)=\lim_{n\to\infty}\eta(n\vx)/n$. By subadditivity we have $\eta(\vx)\ge \nu(\vx)$ for all $\vx\in\dZ^{d}$. It turns out that under assumptions~\ref{ass:1} and \ref{ass:2}, using Alexander's argument (see~\cites{alex93, alex97}) one can prove the following result.

\begin{lem}[see Theorem $4.1$ in~\cite{chatterjee11}]\label{lem:alex}
Assume \ref{ass:1} and \ref{ass:2}. Let $\nu$ and $\eta$ be as defined in \eqref{hw} and \eqref{eta}. Then for any $\chi'>\chi$ there exists $C>0$ such that for all $\vx\in\dZ^{d}\setminus\{\vzero\}$ we have
\[
\nu(\vx)\le \eta(\vx)\le \nu(\vx) + C|\vx|^{\chi'}\log|\vx|. 
\]
\end{lem}

We will use the following result.
\begin{lem}\label{lem:km}
Let $\{X_{i}:i\in \cI\}$ be a finite collection of non-negative random variables such that $\E[\exp(\ga X_{i})]\le C$ for all $i\in \cI$ for some $\ga>0$. Then we have
\[
||\max_{i\in\cI}X_i||_k \le \frac{k}{\ga}\log(2C|\cI|) 
\]
for all $k\ge 1$. 
\end{lem}
\begin{proof}
Fix $k\ge 1$. Let $Y:=\max_{i\in\cI}X_i$ and $Z:=Y^{k-1}/||Y||_{k}^{k-1}$. Clearly we have $||Z||_{1}\le ||Z||_{k/(k-1)}=1$. Moreover we have, by concavity of the logarithm function
\begin{align*}
||Y||_{k}=\E[ZY] &\le \frac{k}{\ga}\E[Z\log(\sum_{i\in\cI}\exp(\ga X_i/k))]\\
&\le \frac{k||Z||_{1}}{\ga}\log(\sum_{i\in\cI}\E\left[\frac{Z}{||Z||_1}\exp(\ga X_i/k)\right])\\
&\le \frac{k||Z||_{1}}{\ga}\log(\sum_{i\in\cI}\frac{1}{||Z||_1}||Z||_{k/(k-1)}\E[\exp(\ga X_i)]^{1/k})\\
&\le \frac{k||Z||_{1}}{\ga}\log(\frac{|\cI|}{||Z||_1}C^{1/k})
\le \frac{1}{\ga}(k\ln(2|\cI|) + \ln C)
\end{align*}
where in the last line we used the fact that $-x\log x\le \log 2$ for all $x\in[0,1]$. 
This completes the proof.
\end{proof}

Now we are ready to prove the results in Section~\ref{sec:pclt}.

\begin{proof}[Proof of Lemma \ref{lem:deltabd}]
We want to bound the moments of the random variable 
\[
\gD:=\gD(R)= \max_{\vx\in B_{L},\vy\in B_{R}} T(\vx,\vy) -\min_{\vx\in B_{L},\vy\in B_{R}} T(\vx,\vy)
\]
where $R$ is the cylinder $[n]\times[-h,h]^{d-1}$, $B_{L}$ is the left boundary wall $\{\mvzero\}\times [-h,h]^{d-1}$, $B_{R}$ is the right boundary wall $\{n\}\times [-h,h]^{d-1}$ and $T(\vx,\vy)$ is the minimum passage time from $\vx$ to $\vy$ inside the cylinder $R$. Note that $h\approx n^{\xi'}$ for some $\xi'>\xi$. Choose $\chi'\in (\chi, 2\xi'-1)$. This is possible since $\chi=2\xi-1$. 

We define $a(\vx,\vy)$ as the unrestricted minimum passage time from $\vx$ to $\vy$. Clearly $a(\vx,\vy)\le T(\vx,\vy)$ and $\E[a(\vx,\vy)]=\nu(\vy-\vx)$. We have
\begin{align*}
0\le \gD &\le  \max_{\vx\in B_{L},\vy\in B_{R}} (T(\vx,\vy)-\nu(\vy-\vx)) +  \max_{\vx\in B_{L},\vy\in B_{R}} (\nu(\vy-\vx)-a(\vx,\vy))\\
&\qquad +  (\max_{\vx\in B_{L},\vy\in B_{R}} \nu(\vy-\vx) -  \min_{\vx\in B_{L},\vy\in B_{R}} \nu(\vy-\vx)  ).
\end{align*}
Denote the three terms appearing in the r.h.s.~by $U,V,Z$ respectively. Note that
\begin{align*}
V\le  Cn^{\chi'}\max_{\vx\in B_{L},\vy\in B_{R}} \frac{|\nu(\vy-\vx)-a(\vx,\vy)|}{|\vx-\vy|^{\chi'}}
\end{align*}
and $|B_{L}||B_{R}|\le Ch^{2(d-1)}\le C n^{C}$.
By assumption~\ref{ass:1} and Lemma~\ref{lem:km} we have 
\[
||V||_{k}\le Ckn^{\chi'}\ln n\le Cn^{2\xi'-1}.
\] 
Now to bound $Z$ we use Lemma~\ref{lem:alex}. We have
\begin{align*}
0\le Z &\le \max_{\vx\in B_{L},\vy\in B_{R}} (\eta(\vy-\vx) + C|\vy-\vx|^{\chi'}\log |\vy-\vx|) -  \min_{\vx\in B_{L},\vy\in B_{R}} \eta(\vy-\vx) \\
&\le Cn^{\chi'}\log n + 2 \max_{\vx\in B_{L},\vy\in B_{R}} |\eta(\vy-\vx)-\eta(n\mve_{1})|.
\end{align*}
Now note that for $\vx\in B_{L},\vy\in B_{R}$ we have $\vy-\vx=n\mve_{1}+z$ where $z\perp \mve_{1}$ and $|z|\le Ch_{n}$. Using assumption \ref{ass:3} we have
\begin{align*}
\max_{\vx\in B_{L},\vy\in B_{R}} |\eta(\vy-\vx)-\eta(n\mve_{1})| 
&= n \cdot \max_{\vx\in B_{L},\vy\in B_{R}} |\eta((\vy-\vx)/n)-\eta(\mve_{1})|\\
& \le Cn(h_n/n)^{2}\le Cn^{2\xi'-1}.
\end{align*}
Now to bound $U$, we divide the boundary into two sets. Define $B_{L}^{o}$ as the boundary part $\{\mvzero\}\times \{z\in\dZ^{d-1}:|z|\le h_{n}/2\}$ and $B_{L}^{b}$ as the boundary part $\{\mvzero\}\times \{z\in\dZ^{d-1}: h_{n}/2<|z|\le h_{n}\}$.  Similarly we define $B_{R}^{o}, B_{R}^{b}$ (o is for center and b is for border). We also define the event 
\[
E(\vx,\vy):=\text{ the unconstrained geodesic from $\vx$ to $\vy$ lies within the cylinder $R$}.
\]
Using Assumption \ref{ass:2} one can easily see that for $\vx\in B_{L}^{o},\vy\in B_{R}^{o}$ we have $\pr(E(\vx,\vy)^{c})\le \exp(-n^{\eps})$ for some $\eps>0$. Thus we have 
\[
\max_{\vx\in B_{L}^{o},\vy\in B_{R}^{o}} (T(\vx,\vy)-\nu(\vy-\vx)) \le \max_{\vx\in B_{L}^{o},\vy\in B_{R}^{o}} (a(\vx,\vy)-\nu(\vy-\vx) + 2n\cdot\ind\{E(\vx,\vy)^{c}\})
\]
and its $k$-th norm is bounded by $Ckn^{\chi'}\log n$.
When either $\vx\in B_{L}^{b}$ or $\vy\in B_{R}^{b}$, we consider the nearest boundary point of $[n/3,2n/3]\times [-h_{n}/2,h_{n}/2]^{d-1}$ to $\vx$ or $\vy$. Call them $\vx'$ and $\vy'$ respectively (if $\vx\in B_{L}^{o}$ we will take $\vx'=\vx$ and similar for $\vy$). Clearly $T(\vx,\vy)\le T(\vx,\vx') + T(\vx',\vy') + T(\vy',\vy)$. As before $T(\vx',\vy')$ will equal $a(\vx',\vy')$ with high probability. Also note that $0\le \nu(\vx'-\vx) + \nu(\vy'-\vx') + \nu(\vy-\vy')-\nu(\vy-\vx)\le Cn^{2\xi'-1}$ for all such $\vx,\vy$. Thus we need to bound the $k$-th norm of $\max_{\vx,\vx'}(T(\vx,\vx')-\nu(\vx'-\vx))$. But considering the diagonal direction (for which Assumption \ref{ass:1} and \ref{ass:2} are also valid) and using the event that the unrestricted geodesic stays within the corresponding cylinder of radius $|\vx'-\vx|^{\xi'}$ and Using Lemma~\ref{lem:km} we get the bound that $||U||_{k}\le Cn^{2\xi'-1}$. Combining everything we have
\[
||\gD(R)||_{k}\le Cn^{2\xi'-1}=Ch^{2}/n.
\]
\end{proof}
\begin{proof}[Proof of Lemma \ref{lem:clbd}]
We will first prove the variance upper bound. Under assumption \ref{ass:1} and \ref{ass:3} one can easily check that for any $\xi'>\xi$ and $\chi'>\chi>\chi''$, there exist constants $c,C>0$ such that
\[
cn^{2\chi''}\le \var(a_{n}(n^{\xi'})) \le Cn^{2\chi'}.
\]
Combining with  Lemma~\ref{lem:deltabd} we get 
\begin{align}\label{eq:vbd2}
cn^{2\chi''}\le \var(X(n,n^{\xi'})) \le Cn^{2\chi'}.
\end{align}
Fix $\eps\in (0,(2\xi)^{-1})$. Define $\xi'>\xi$ such that $1/\xi'=1/\xi-\eps$. Moreover define $\chi'$ such that
\begin{align}\label{eq:cx}
\theta':=\frac{1-2\chi'}{\xi'}=\theta-4\eps=\frac{1-2\chi}{\xi}-4\eps.
\end{align}
Note that $\chi'>\chi$ as \eqref{eq:cx} implies that $(2\chi'-1)(1-\eps\xi)=(2\chi-1)+4\eps\xi$ or $2(\chi'-\chi)(1-\eps\xi)=\eps\xi(2\chi+3)>0$. For simplicity we will always take $h_n$ of the form $n^{\gc}$ for some $\gc\in (0,\xi']$. Define $\gc_{1}=\xi'$. From \eqref{eq:vbd2} we have
\begin{align}
\var(X(n,n^{\gc_1}))\le Cn(n^{\gc_1})^{-\theta'}
\end{align}
for large enough $n$.

We will use an induction argument to prove the upper bound. Suppose that for some $\gc>0$ we have
\begin{align}\label{eq:ass1}
\var(X(n,n^{\gc}))\le Cn(n^{\gc})^{-\theta'}.
\end{align}
for all $n$ large enough. We consider the cylinder $[n]\times[-n^{\gc'},n^{\gc'}]$ where $\gc'<\gc$ and divide it into consecutive big and small cylinders of length $\ell_{1}:=n^{\gc'/\gc}$ and $\ell_{2}:=n^{\gc'/\xi'}$ respectively. Number of such cylinders will be $m\approx n^{1-\gc'/\gc}$. Using \eqref{eq:vsum} and Lemma~\ref{lem:deltabd} we have
\begin{align}
\sqrt{\var(X(n,n^{\gc'}))}
&\le C\sqrt{m\var(X(n^{\gc'/\gc},n^{\gc'}))+ m\var(X(n^{\gc'/\xi'},n^{\gc'}))} + Cm\cdot n^{(2-1/\xi')\gc'}\notag\\
&\le C\sqrt{n^{1-\theta'\gc'}} + n^{1-\gc'/\gc + (2-1/\xi')\gc'}\label{eq:vrec}
\end{align}
where in the last line we have used  \eqref{eq:ass1}. Thus the variance upper bound
\[
{\var(X(n,n^{\gc'}))} \le C'n(n^{\gc'})^{-\theta'}
\]
will hold (with a different constant $C'$) if we have
\begin{align}\label{eq:gcond}
2(1-\gc'/\gc + (2-1/\xi')\gc') &\le 1-\theta'\gc'
\text{ or } 1/\gc' - 2/\gc \le - (\theta'+4-2/\xi').
\end{align}
Define $\gl:=\theta'+4-2/\xi'$.
Putting the values of $\theta',\xi'$ and using the fact that $\chi=2\xi-1$ we have 
\[
\gl= (1-2\chi)/\xi - 4\eps +4 - 2/\xi + 2\eps = 1/\xi-2\eps<1/\xi'. 
\]
Thus if 
\[
\frac{1}{\gc'}-\gl\le 2\left(\frac{1}{\gc}-\gl\right)
\]
and the variance upper bound holds for $\gc$, then the variance upper bound also holds for $\gc'$. Now starting with $\gc_{1}=\xi'$ for which the variance upper bound holds, we can see that the upper bound holds for $\gc$ (with a constant $C$ depending on $t$) if 
\[
\frac{1}{\gc}-\gl\le 2^{t}\left(\frac{1}{\xi'}-\gl\right)= 2^{t} \eps
\]
for some positive integer $t\ge 1$. By taking $t$ large we have the result. 
 
 To prove the upper bound for $k$-th central moment, we use the following result from Lata\l a~\cite{latala97}.
 
\begin{lem}[Theorem $2$ in Lata\l a~\cite{latala97}]\label{lem:latala}
If $k\ge 1$ and $X_{1},X_{2},\ldots$ are i.i.d.~mean zero random variables then we have
\[
||X_{1}+\cdots+X_{n}||_{k} \sim \sup\left\{\frac{k}{s}\left( \frac{n}{k}\right)^{1/s}||X_{1}||_{s}: \max\{2,k/n\}\le s\le k\right\}.
\]
\end{lem}

Using assumption \ref{ass:1}, \ref{ass:3} and Lemma~\ref{lem:deltabd} one can easily check that for any $k\ge 2, \xi'>\xi$ and $\chi'>\chi>\chi''$, there exist constants $c,C>0$ such that
\[
cn^{\xi''}\le ||X(n,n^{\xi'})-\E[X(n,n^{\xi'})]||_{k} \le Cn^{\xi'}.
\]

We use induction over $k$. For $k=2$ the $k$-th moment upper bound is true. Now note that if all the moments $||X_{1}||_{i}$ for $2\le i\le k$ are upper bounded by $Cn^{\chi'}$, then from Lemma~\ref{lem:latala} we have
\[
||X_{1}+\cdots+X_{\ell}||_{k} \le C_{k}\ell^{1/2}n^{\chi'}.
\]
Moreover, from equation \eqref{eq:kmbd}, for the $k$-th central moment (similar to \eqref{eq:vrec}) we have 
\[
||X(n,n^{\gc'})-\E[X(n,n^{\gc'})]||_{k} \le C||X_{1}+\cdot+X_{m}||_{k} + Cm\cdot n^{(2-1/\xi')\gc'}
\]
where $X_{i}\equald X(n^{\gc'/\gc},n^{\gc'})-E[X(n^{\gc'/\gc},n^{\gc'})]$ are i.i.d., and this is sufficient to run the induction. We leave the proof details, which is similar to the variance upper bound,  to the interested reader. 

Now we move on to the proof of the variance lower bound under the assumption \ref{ass:4} and $\chi>0$. Here also we will take $h_n$ of the form $n^{\gc}$ for some $\gc\in (0,\xi']$. Suppose for some $\gc_{0}<\xi$ the variance lower bound does not hold for $h_{n}=n^{\gc_{0}}$ so that there exists $\eps'>0$ such that 
\begin{align}\label{eq:lbd1}
\var(X(n,n^{\gc_{0}})) \le  cn(h_{n})^{-\theta-\eps'}
\end{align}
for an increasing sequence of $n$. We will use the same idea used in the variance upper bound. But instead of estimating the variance of thin cylinder from thick cylinders, here we are estimating the variance of the thick cylinder from thin cylinders. If \eqref{eq:lbd1} holds for $\gc'$ instead of $\gc_{0}$, using the same idea used in the variance upper bound, we can easily see that it will hold for $\gc>\gc'$ if 
\begin{align*}
1/\gc' - 2/\gc &< - (\theta +\eps' +4-2/\xi')=-(1/\xi +\eps' +2\eps)\\
\text{ or } 1/\gc -\gl &> 1/2\cdot (1/\gc'-\gl)
\end{align*}
where $\gl=1/\xi +\eps' +2\eps$. If $\gc_{0}<1/\gl$ using finitely many steps we can show that \eqref{eq:lbd1} holds for $\gc$ smaller than but arbitrary close to $1/\gl$. Now for $\gc$ very close to $1/\gl$ or bigger than that, the variance upper bound becomes
\[
\var(X(n,n^{\gc})) \le  cn^{1-\gc(\theta+\eps')}.
\]
Now note that for $\gc=1/\gl$ we have 
\[
1-\gc(\theta+\eps') = 1-\frac{\xi}{1+(\eps'+2\eps)\xi}\left(\frac{1-2\chi}{\xi}+\eps'\right) = \frac{2\eps\xi + 2\chi}{1+(\eps'+2\eps)\xi}
\]
which is strictly smaller than $2\chi$ for $\eps$ small enough. Thus in finitely many steps we get a variance upper bound 
\[
\var(X(n,n^{\gc})) \le  cn^{2\chi'}
\]
where $\gc<\xi$ and $\chi'<\chi$. But under Assumption~\ref{ass:4} we have 
\[
\var(X(n,n^{\xi'})) \le \var(X(n,n^{\gc})) \le  cn^{2\chi'}
\]
for all $\xi'\ge \xi$ and for large enough $n$. This gives a contradiction to \eqref{eq:vbd2} and we are done. 
\end{proof}

\begin{proof}[Proof of Theorem \ref{thm:mainht}]
We will use the same notations as in Section \ref{sec:clt_ht}. To prove the CLT we use the two type blocking with the big blocks having length $\ell_{1}\approx n^{1-\gb}$ and small blocks having lengths $\ell_{2}\approx h_{n}^{1/\xi'}$ where $\xi'>\xi$ is fixed. Number of such cylinders is $2m+1\approx n^{\gb}$. In the proof $\gb>0$ will be very small but fixed. From \eqref{eq:kmbd} we have
\begin{align}
\norm{(a_n(h_n)-\E[a_n(h_n)])- \sum_{i=1}^{2m+1}(X_{i}-\E[X_{i}])}_{2} \le 2m\cdot ||\gD(R)||_{2}
\end{align}
where $R$ is the cylinder $[\ell_{2}]\times[-h_{n},h_{n}]^{d-1}$. Define 
\[
\gs_{n}^{2}(h_{n}):= \var\left(\sum_{i=1}^{2m+1}X_{i}\right) = (m+1)\var(T(\ell_{1},h_{n}))+m\var(T(\ell_{2},h_{n})).
\]
From equation \eqref{eq:vsum} we have
\begin{align}
|\sqrt{\var(a_{n}(h_{n}))}-\gs_{n}(h_{n})|\le 2m||\gD(R)||_{2}
\end{align}
Moreover, from Lemma~\ref{lem:deltabd} we have 
\[
||\gD(R)||_{2} \le \frac{Ch_n^2}{\ell_{2}} \le Ch_{n}^{2-1/\xi'}.
\]
If we can show that 
\begin{align}\label{eq:2show}
m||\gD(R)||_{2} \ll \gs_{n}(h_{n})
\end{align}
we will have
\[
\norm{\frac{a_n(h_n)-\E[a_n(h_n)]}{\sqrt{\var(a_{n}(h_n))}}- \frac{\sum_{i=1}^{2m+1}(X_{i}-\E[X_{i}])}{\gs_{n}(h_n)}}_{2} \goesto 0
\]
as $n\to\infty$. 

Now when $\chi>0$ using the variance lower bound from Lemma~\ref{lem:clbd} and choosing $\gb,\xi'-\xi$ sufficiently small, one can show that \ref{eq:2show} holds for $h_{n}=o(n^{\ga})$ with $\ga<\xi$ (see the discussion before Theorem~\ref{thm:mainht}). When $\chi=0$, using the variance lower bound from Lemma~\ref{lem:lbd}
\[
\var(X(n,h)) \ge cnh^{-(d-1)}
\]
and the fact that $\xi=1/2$, it follows that \eqref{eq:2show} holds for $h_{n}=o(n^{\ga})$ with $\ga<1/(d-1)$.

The rest of the proof of CLT can be completed using Lyapounov's condition and  the same recursion idea used in the proof of Theorem~\ref{thm:maingr}. However, when $\chi>0$, it is possible to prove the CLT for $\sum_{i=1}^{2m+1}X_{i}$ directly using the $k$-th central moment bound from Lemma~\ref{lem:clbd} as in that case for $k>1$ we have
\[
\frac{(m+1)||X_{1}-\E[X_{1}]||_{2k}^{2k}+m||X_{2}-\E[X_{2}]||_{2k}^{2k}}{((m+1)\var(X_{1})+m\var(X_{2}))^{k}} \le \frac{Cm (nh_{n}^{-\theta+\eps}/m)^k}{(nh_{n}^{-\theta-\eps})^k}\to 0
\] 
as $n\to\infty$ for $\eps$ small enough. 
\end{proof}


\section{Numerical results}\label{sec:num}
In this section we report on some numerical simulation results which support Conjecture~\ref{conj:onethird} and \ref{conj:varbd}. We consider two-dimensional rectangles $\{0,1,\ldots,n\}\times \{-h_n,\ldots, h_{n}\}$ with $h_{n}=n^{\alpha}$ for $h_n$ ranging between $30$ to $60$ and $\alpha$ ranging within the set $\{2/3, 1/2, 2/5$, $1/3\}$.  For the edge weight distribution we take Bernoulli$(p)$ for different values of $p$. For each configuration we simulate $1000$ observations for $a_{n}(h_{n})$ to estimate the variance and use $1000$ estimates for the variance per configuration to estimate the parameters.  

We assume that there are two constants $\beta,\gc >0$  depending only on the distribution of edge weights such that
\[
\var(a_{n}(h_{n}))\approx \beta n h_{n}^{-\gc}
\] 
for $h_{n}\le n^{2/3}$. Note that we have the rigorous result that $\gc\in [0,1]$ if it exists. However it is not clear how to define the approximation properly. Our conjecture is that $\gc$ exists in some appropriate sense (for example the ratio of the logarithms of both sides are bounded) and satisfies the following:

\begin{conj}
In two dimension, we have 
\[\gc=1/2\]
 when $h_{n}=\Theta(n^{\ga})$ and $\ga\le 2/3$.
\end{conj}

To estimate the numbers $\gb,\gc$  we use the simple  linear regression model 
$$
\log \var(a_n(h_n)) = \log\beta + \log n -\gc \log(h_n) +\text{ Gaussian error}
$$
and least square estimates. 
In figure~\ref{fig:gamma} the estimated values of $\gc$ are plotted against $p$ for different values  $\ga$, which shows that $\gc$ is close to $1/2$ for all values of $p$.

\begin{figure}[htbf]
\begin{center}
   \includegraphics[height=2.2in,width=4.5in]{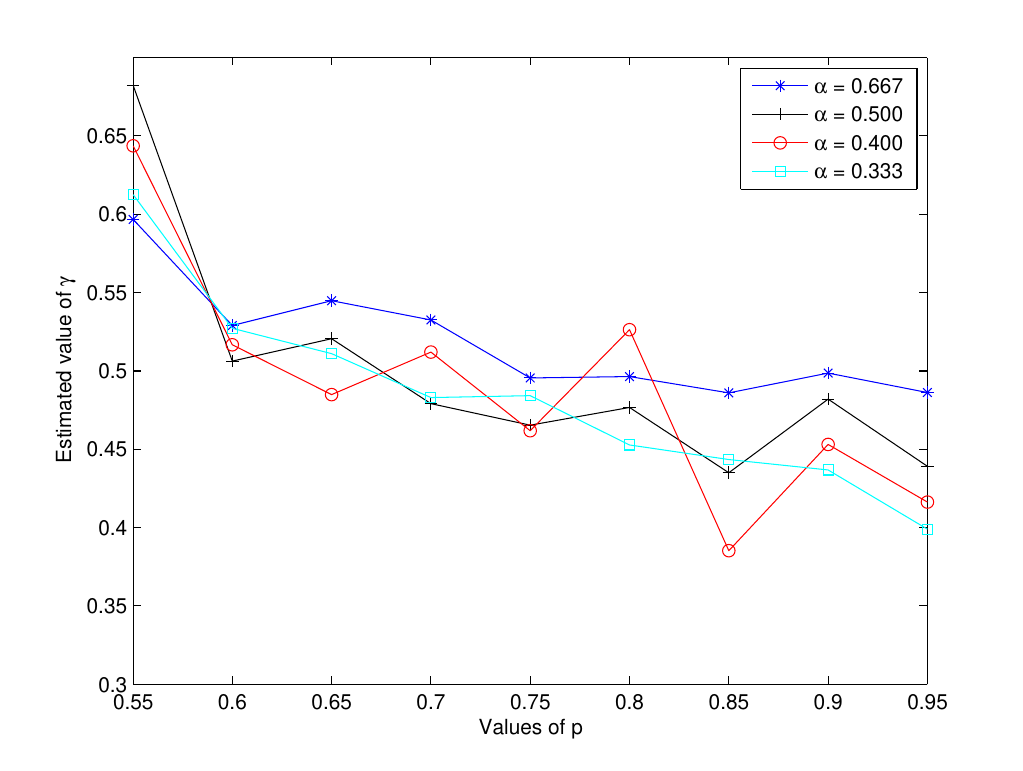}\\ 
\end{center}
   \caption{Plot of estimated values of  $\gc$ vs. $p$ for different values of $\ga$.}
   \label{fig:gamma}
\end{figure}

Figure~\ref{fig:qq} shows QQ plots based on the above simulation data for $a_{n}(h_{n})$ for $n=h_{n}^{2}=55$ against an appropriately fitted normal distribution, supporting the conjecture of asymptotic  normality. 

\begin{figure}[htbf] 
\begin{minipage}{1\columnwidth}
\begin{center}
   \includegraphics[width=1.75in]{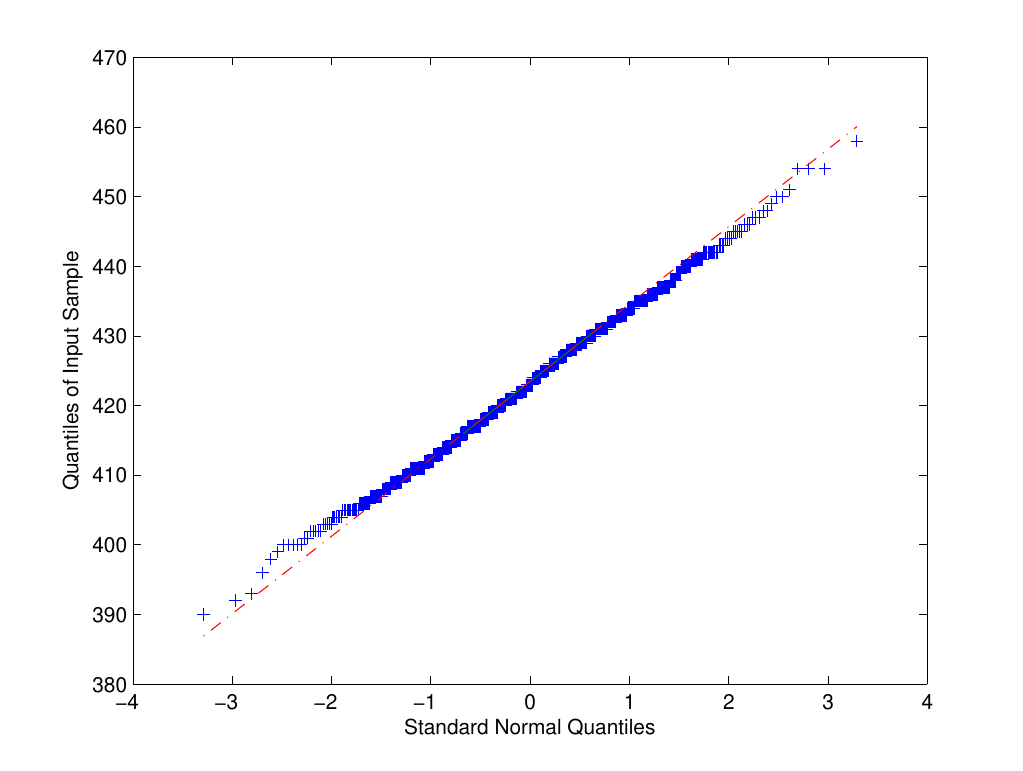}
   \includegraphics[width=1.75in]{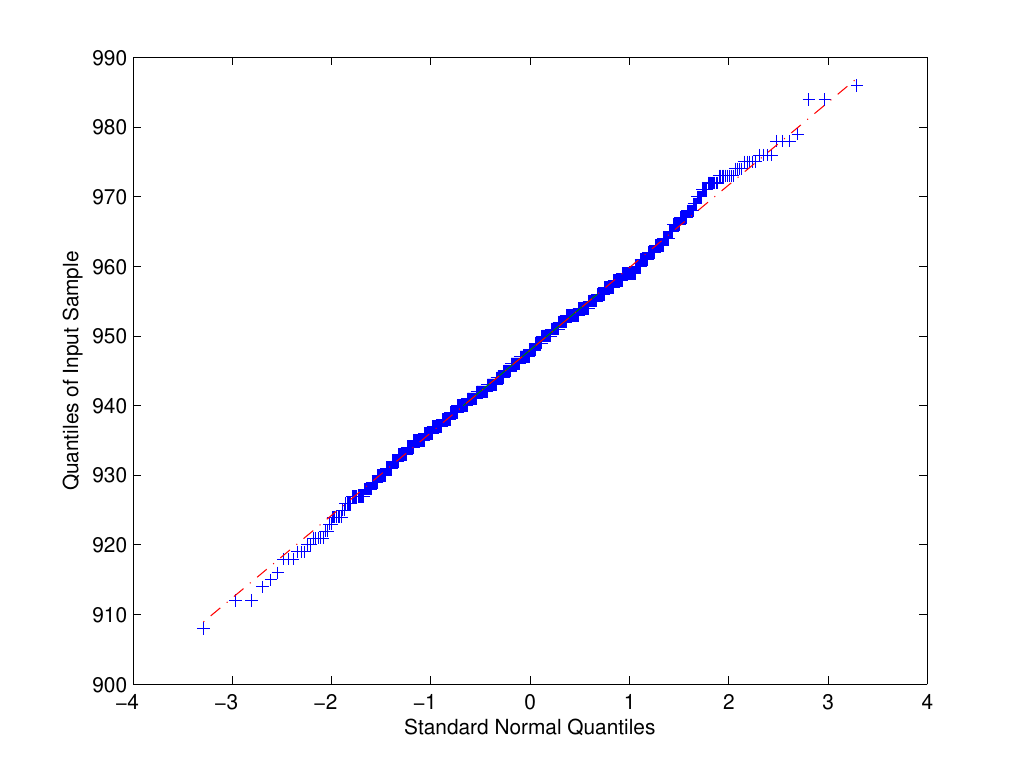}
\end{center}
\end{minipage}

\begin{minipage}{1\columnwidth}
\begin{center}
  \includegraphics[width=1.75in]{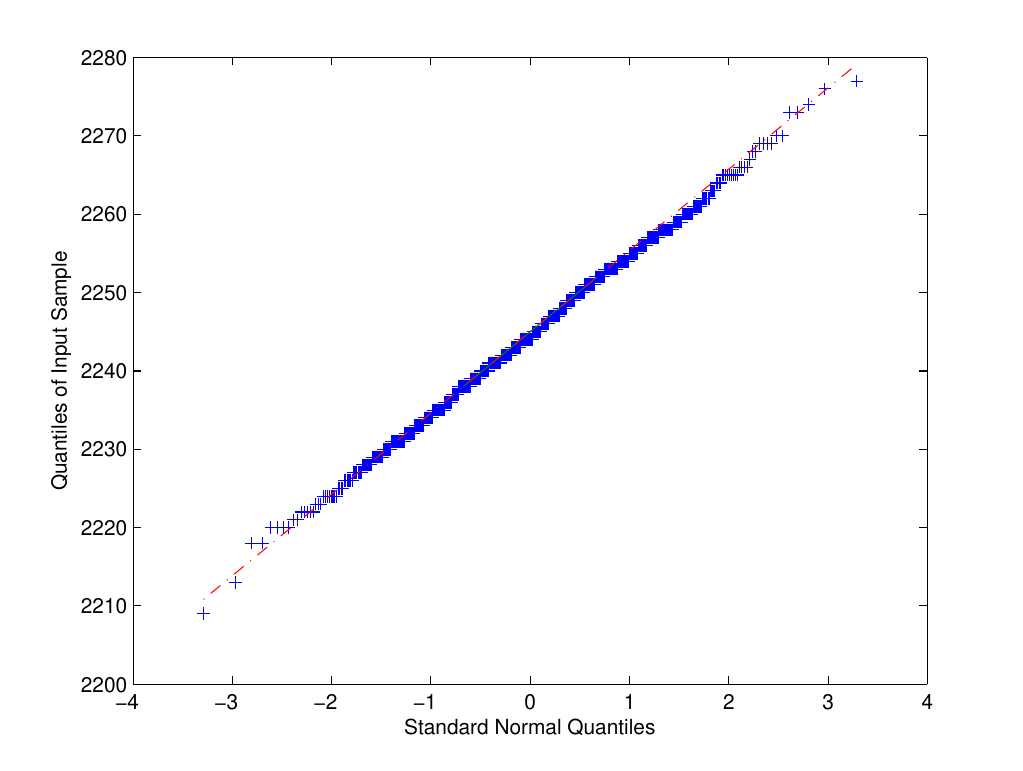}
   \includegraphics[width=1.75in]{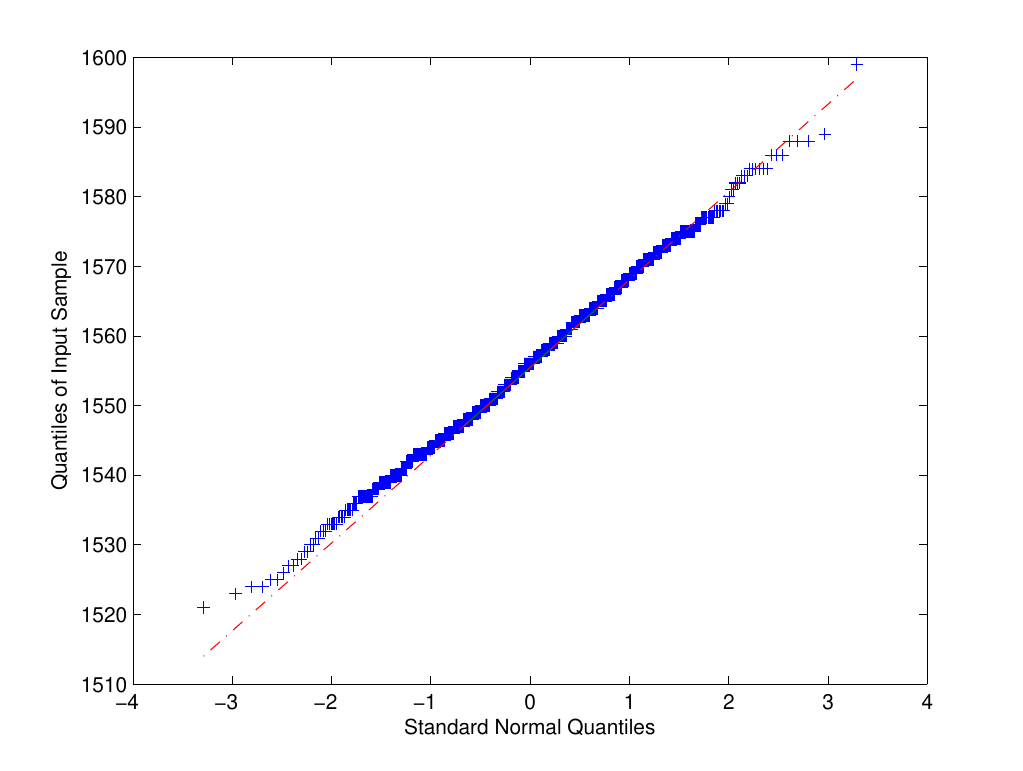}
\end{center}
\end{minipage}
   \caption[QQ plots based on simulation data for $a_{n}(n^{1/2})$ for $n=3000$]{QQ plots based on simulation data for $a_n(n^{1/2})$ for $n=3000$ for Bernoulli$(p)$ edge weights, $p=0.6,0.7,0.8,0.9$ in clockwise direction starting from top left.}
   \label{fig:qq}
\end{figure}

\section{Acknowledgments} 
\label{sec:acknowledgments}
 	The authors would like to thank Itai Benjamini for initiating the investigation by suggesting that a CLT may hold for cylinders with fixed diameter. They are thankful to Antonio Auffinger and Oren Louidor for helping with the computer simulation and to the anonymous referee for several helpful comments that improved the presentation of the paper. 
         

\begin{bibdiv}
\begin{biblist}

\bib{alex93}{article}{
      author={Alexander, Kenneth~S.},
       title={A note on some rates of convergence in first-passage
  percolation},
        date={1993},
        ISSN={1050-5164},
     journal={Ann. Appl. Probab.},
      volume={3},
      number={1},
       pages={81\ndash 90},
  url={http://links.jstor.org/sici?sici=1050-5164(199302)3:1<81:ANOSRO>2.0.CO;2-D&origin=MSN},
}

\bib{alex97}{article}{
      author={Alexander, Kenneth~S.},
       title={Approximation of subadditive functions and convergence rates in
  limiting-shape results},
        date={1997},
        ISSN={0091-1798},
     journal={Ann. Probab.},
      volume={25},
      number={1},
       pages={30\ndash 55},
         url={http://dx.doi.org/10.1214/aop/1024404277},
      review={\MR{MR1428498 (98f:60203)}},
}

\bib{am11}{article}{
      author={Auffinger, Antonio},
      author={Damron, Michael},
       title={A simplified proof of the relation between scaling exponents in
  first-passage percolation},
        date={2011},
      eprint={http://arxiv.org/abs/1109.0523},
}

\bib{bs05}{article}{
      author={Baik, Jinho},
      author={Suidan, Toufic~M.},
       title={A {GUE} central limit theorem and universality of directed first
  and last passage site percolation},
        date={2005},
        ISSN={1073-7928},
     journal={Int. Math. Res. Not.},
      number={6},
       pages={325\ndash 337},
      review={\MR{MR2131383 (2006c:60025)}},
}

\bib{bks03}{article}{
      author={Benjamini, Itai},
      author={Kalai, Gil},
      author={Schramm, Oded},
       title={First passage percolation has sublinear distance variance},
        date={2003},
        ISSN={0091-1798},
     journal={Ann. Probab.},
      volume={31},
      number={4},
       pages={1970\ndash 1978},
      review={\MR{MR2016607 (2005b:60251)}},
}

\bib{bl68}{book}{
      author={Billingsley, Patrick},
       title={Convergence of probability measures},
   publisher={John Wiley \& Sons Inc.},
     address={New York},
        date={1968},
      review={\MR{MR0233396 (38 \#1718)}},
}

\bib{bm05}{article}{
      author={Bodineau, Thierry},
      author={Martin, James},
       title={A universality property for last-passage percolation paths close
  to the axis},
        date={2005},
        ISSN={1083-589X},
     journal={Electron. Comm. Probab.},
      volume={10},
       pages={105\ndash 112 (electronic)},
      review={\MR{MR2150699 (2006a:60189)}},
}

\bib{bblm05}{article}{
      author={Boucheron, St{\'e}phane},
      author={Bousquet, Olivier},
      author={Lugosi, G{\'a}bor},
      author={Massart, Pascal},
       title={Moment inequalities for functions of independent random
  variables},
        date={2005},
        ISSN={0091-1798},
     journal={Ann. Probab.},
      volume={33},
      number={2},
       pages={514\ndash 560},
      review={\MR{MR2123200 (2006a:60024)}},
}

\bib{chatterjee11}{article}{
      author={Chatterjee, Sourav},
       title={The universal relation between scaling exponents in first-passage
  percolation},
        date={2011},
      eprint={http://arxiv.org/abs/1105.4566},
}

\bib{ccd86}{article}{
      author={Chayes, J.~T.},
      author={Chayes, L.},
      author={Durrett, R.},
       title={Critical behavior of the two-dimensional first passage time},
        date={1986},
        ISSN={0022-4715},
     journal={J. Statist. Phys.},
      volume={45},
      number={5-6},
       pages={933\ndash 951},
      review={\MR{MR881316 (88f:60175)}},
}

\bib{cd81}{article}{
      author={Cox, J.~Theodore},
      author={Durrett, Richard},
       title={Some limit theorems for percolation processes with necessary and
  sufficient conditions},
        date={1981},
        ISSN={0091-1798},
     journal={Ann. Probab.},
      volume={9},
      number={4},
       pages={583\ndash 603},
      review={\MR{MR624685 (82k:60208)}},
}

\bib{gtw01}{article}{
      author={Gravner, Janko},
      author={Tracy, Craig~A.},
      author={Widom, Harold},
       title={Limit theorems for height fluctuations in a class of discrete
  space and time growth models},
        date={2001},
        ISSN={0022-4715},
     journal={J. Statist. Phys.},
      volume={102},
      number={5-6},
       pages={1085\ndash 1132},
      review={\MR{MR1830441 (2002d:82065)}},
}

\bib{gk84}{article}{
      author={Grimmett, Geoffrey},
      author={Kesten, Harry},
       title={First-passage percolation, network flows and electrical
  resistances},
        date={1984},
        ISSN={0044-3719},
     journal={Z. Wahrsch. Verw. Gebiete},
      volume={66},
      number={3},
       pages={335\ndash 366},
      review={\MR{MR751574 (86d:60128)}},
}

\bib{hw65}{incollection}{
      author={Hammersley, J.~M.},
      author={Welsh, D. J.~A.},
       title={First-passage percolation, subadditive processes, stochastic
  networks, and generalized renewal theory},
        date={1965},
   booktitle={Proc. {I}nternat. {R}es. {S}emin., {S}tatist. {L}ab., {U}niv.
  {C}alifornia, {B}erkeley, {C}alif},
   publisher={Springer-Verlag},
     address={New York},
       pages={61\ndash 110},
      review={\MR{MR0198576 (33 \#6731)}},
}

\bib{kj00}{article}{
      author={Johansson, Kurt},
       title={Shape fluctuations and random matrices},
        date={2000},
        ISSN={0010-3616},
     journal={Comm. Math. Phys.},
      volume={209},
      number={2},
       pages={437\ndash 476},
      review={\MR{MR1737991 (2001h:60177)}},
}

\bib{kj01}{article}{
      author={Johansson, Kurt},
       title={Discrete orthogonal polynomial ensembles and the {P}lancherel
  measure},
        date={2001},
        ISSN={0003-486X},
     journal={Ann. of Math. (2)},
      volume={153},
      number={1},
       pages={259\ndash 296},
      review={\MR{MR1826414 (2002g:05188)}},
}

\bib{kpz86}{article}{
      author={Kardar, M.},
      author={Parisi, G.},
      author={Zhang, Y.C.},
       title={{Dynamic scaling of growing interfaces}},
        date={1986},
     journal={Physical Review Letters},
      volume={56},
      number={9},
       pages={889\ndash 892},
}

\bib{kes86}{incollection}{
      author={Kesten, Harry},
       title={Aspects of first passage percolation},
        date={1986},
   booktitle={\'{E}cole d'\'et\'e de probabilit\'es de {S}aint-{F}lour,
  {XIV}---1984},
      series={Lecture Notes in Math.},
      volume={1180},
   publisher={Springer},
     address={Berlin},
       pages={125\ndash 264},
      review={\MR{MR876084 (88h:60201)}},
}

\bib{kes93}{article}{
      author={Kesten, Harry},
       title={On the speed of convergence in first-passage percolation},
        date={1993},
        ISSN={1050-5164},
     journal={Ann. Appl. Probab.},
      volume={3},
      number={2},
       pages={296\ndash 338},
      review={\MR{MR1221154 (94m:60205)}},
}

\bib{nz97}{article}{
      author={Kesten, Harry},
      author={Zhang, Yu},
       title={A central limit theorem for ``critical'' first-passage
  percolation in two dimensions},
        date={1997},
        ISSN={0178-8051},
     journal={Probab. Theory Related Fields},
      volume={107},
      number={2},
       pages={137\ndash 160},
         url={http://dx.doi.org/10.1007/s004400050080},
      review={\MR{MR1431216 (97m:60151)}},
}

\bib{ks91}{article}{
      author={Krug, J.},
      author={Spohn, H.},
       title={{Kinetic roughening of growing surfaces}},
        date={1991},
     journal={Solids far from equilibrium},
       pages={479\ndash 582},
}

\bib{latala97}{article}{
      author={Lata{\l}a, Rafa{\l}},
       title={Estimation of moments of sums of independent real random
  variables},
        date={1997},
        ISSN={0091-1798},
     journal={Ann. Probab.},
      volume={25},
      number={3},
       pages={1502\ndash 1513},
         url={http://dx.doi.org/10.1214/aop/1024404522},
}

\bib{lnp96}{article}{
      author={Licea, C.},
      author={Newman, C.~M.},
      author={Piza, M. S.~T.},
       title={Superdiffusivity in first-passage percolation},
        date={1996},
        ISSN={0178-8051},
     journal={Probab. Theory Related Fields},
      volume={106},
      number={4},
       pages={559\ndash 591},
         url={http://dx.doi.org/10.1007/s004400050075},
      review={\MR{MR1421992 (98a:60151)}},
}

\bib{np95}{article}{
      author={Newman, Charles~M.},
      author={Piza, Marcelo S.~T.},
       title={Divergence of shape fluctuations in two dimensions},
        date={1995},
        ISSN={0091-1798},
     journal={Ann. Probab.},
      volume={23},
      number={3},
       pages={977\ndash 1005},
      review={\MR{MR1349159 (96g:82052)}},
}

\bib{pp94}{article}{
      author={Peres, Yuval},
      author={Pemantle, Robin},
       title={{Planar first-passage percolation times are not tight}},
        date={1994},
     journal={Probability and phase transition (G. Grimmett, ed.)},
       pages={261\ndash 264},
}

\bib{rson73}{article}{
      author={Richardson, Daniel},
       title={Random growth in a tessellation},
        date={1973},
     journal={Proc. Cambridge Philos. Soc.},
      volume={74},
       pages={515\ndash 528},
      review={\MR{MR0329079 (48 \#7421)}},
}

\bib{rose70}{article}{
      author={Rosenthal, Haskell~P.},
       title={On the subspaces of {$L^{p}$} {$(p>2)$} spanned by sequences of
  independent random variables},
        date={1970},
        ISSN={0021-2172},
     journal={Israel J. Math.},
      volume={8},
       pages={273\ndash 303},
      review={\MR{MR0271721 (42 \#6602)}},
}

\bib{sw78}{book}{
      author={Smythe, R.~T.},
      author={Wierman, John~C.},
       title={First-passage percolation on the square lattice},
      series={Lecture Notes in Mathematics},
   publisher={Springer},
     address={Berlin},
        date={1978},
      volume={671},
        ISBN={3-540-08928-4},
      review={\MR{MR513421 (80a:60135)}},
}

\bib{suidan06}{article}{
      author={Suidan, Toufic},
       title={A remark on a theorem of {C}hatterjee and last passage
  percolation},
        date={2006},
        ISSN={1751-8113},
     journal={J. Phys. A},
      volume={39},
      number={28},
       pages={8977\ndash 8981},
      review={\MR{MR2240468 (2007d:82040)}},
}

\bib{tala95}{article}{
      author={Talagrand, Michel},
       title={Concentration of measure and isoperimetric inequalities in
  product spaces},
        date={1995},
        ISSN={0073-8301},
     journal={Inst. Hautes \'Etudes Sci. Publ. Math.},
      number={81},
       pages={73\ndash 205},
      review={\MR{MR1361756 (97h:60016)}},
}

\bib{zhang08}{article}{
      author={Zhang, Yu},
       title={Shape fluctuations are different in different directions},
        date={2008},
        ISSN={0091-1798},
     journal={Ann. Probab.},
      volume={36},
      number={1},
       pages={331\ndash 362},
         url={http://dx.doi.org/10.1214/009117907000000213},
      review={\MR{MR2370607 (2009c:60273)}},
}
\end{biblist}
\end{bibdiv}

\end{document}